\documentclass{louloupart1}
\usepackage{newtxtext,newtxmath}
\usepackage{pinlabel}
\usepackage{graphicx}
\usepackage[all]{xy}
\usepackage[sorting=nty,style=alphabetic]{biblatex}
\title[Smallest quotients and profinite rigidity of irreducible spherical type Artin groups]{Smallest quotients and profinite rigidity of\\ irreducible spherical type Artin groups}
\author[F. Gavazzi]{Federica Gavazzi}
\givenname{Federica}
\surname{Gavazzi}
\address{Federica Gavazzi, Department of Mathematics, Heriot-Watt University and Maxwell Institute for Mathematical Sciences, Edinburgh, UK.}
\email{F.Gavazzi@hw.ac.uk}
\author[I. Haladjian]{Igor Haladjian}
\givenname{Igor}
\surname{Haladjian}
\address{Igor Haladjian, Institut Denis Poisson, Université de Tours, UMR 7013, 37000 Tours, France.}
\email{igor.haladjian@univ-tours.fr}
\author[L. Paris]{Luis Paris}
\givenname{Luis}
\surname{Paris}
\address{Luis Paris, Universit\'e Bourgogne Europe, CNRS, IMB, UMR 5584, 21000 Dijon, France.}
\email{lparis@u-bourgogne.fr}

\newtheorem{thm}{Theorem}[section]
\newtheorem{lem}[thm]{Lemma}
\newtheorem{prop}[thm]{Proposition}
\newtheorem{corl}[thm]{Corollary}

\theoremstyle{definition}

\newtheorem{rem}[thm]{Remark}

\theoremstyle{definition}
\newtheorem{defn}[thm]{Definition}

\theoremstyle{definition2}
\newtheorem*{acknow}{Acknowledgments}

\newtheorem*{discloAI}{Disclosure of AI tools}

\numberwithin{equation}{section}

\newcommand{\llangle}{\langle\!\langle}
\newcommand{\rrangle}{\rangle\!\rangle}
\newcommand{\pg}{\mathrm{PerfectGroup}}

\makeatletter
\renewcommand{\thefigure}{\ifnum \c@section>\z@ \thesection.\fi
 \@arabic\c@figure}
\@addtoreset{figure}{section}
\makeatother

\begin{document}

\def\A{{\rm A}} 
\def \W{{\rm W}}
\def \M{{\rm M}}
\def\id{{\rm id}} \def\R{\mathbb R} 
\def \Prod{\rm Prod}
\def \B {{\mathcal{B}}}
\def \F {{\mathcal{F}}}
\def \C {{\mathcal{C}}}
\def \S {{\mathfrak{S}}}
\def \N {{\mathcal{N}}}
\def \ab {{\mathrm{ab}}}


\begin{abstract}
We prove that irreducible spherical type Artin groups are profinitely rigid within the class of all spherical type Artin groups. As part of the proof, building on ideas of Kolay for the $n$-strand braid group, we compute the smallest non-abelian quotient of every spherical type Artin group. In particular, given an irreducible spherical type Artin group, its smallest non-abelian quotient is always that of the corresponding Coxeter group, except for dihedral Artin groups of weight divisible by $4$, which all have $\mathfrak S_3$ as a non-abelian quotient. Parts of the proof also rely on the computation of the cohomological dimension of the profinite completion of certain Artin groups.



\end{abstract}

\maketitle


\section{Introduction}

Let $S$ be a finite set. A \textit{Coxeter matrix} on $S$ is a square, symmetric matrix $\M=(m_{s,t})_{s,t\in S}$ with $m_{s,t}\in \mathbb{N}_{\geq 1}\cup \{\infty\}$ such that $m_{s,s}=1$ for all $s\in S$, and $m_{s,t}\geq 2$ for all $s\neq t$ in $S$. Such a matrix is usually encoded by a simplicial labeled graph $\Gamma$, called a \textit{Coxeter graph}, defined as follows. The vertex set of $\Gamma$ is the set $S$, and two distinct vertices $s,t$ are joined by and edge whenever $m_{s,t}\geq 3$, and the edge is labeled by $m_{s,t}$ if $m_{s,t}\geq 4$ or $m_{s,t}=\infty$.

Given a Coxeter graph $\Gamma$ on a finite set $S$, the \textit{Artin group} associated with $\Gamma$, denoted by $\A[\Gamma]$, is the group with the presentation
\[
\A[\Gamma]=\langle S\mid \underbrace{s\,t\,s\;\cdots }_{m_{s,t}\; \text{letters} }=\underbrace{t\,s\,t\;\cdots }_{m_{s,t}\; \text{letters} },\; \text{for } s,t\in S \;\text{such that } s\neq t, m_{s,t}\neq \infty \rangle\,.
\]
The \textit{Coxeter group} associated with $\Gamma$, denoted by $\W[\Gamma]$, is the quotient of $\A[\Gamma]$ by the relations $s^2=\id$ for all $s\in S$. When $\W[\Gamma]$ is a finite group, the Coxeter graph $\Gamma$, the Coxeter group $\W[\Gamma]$, and the Artin group $\A[\Gamma]$ are called \textit{of spherical type}.

Artin groups were introduced by Tits in \cite{Tit66} as extensions of Coxeter groups, but their systematic study began in the 1970s with the works of Brieskorn \cite{Bri71,Bri73}, Brieskorn--Saito \cite{BS72}, and Deligne \cite{Deligne}. These groups have attracted renewed interest in recent years, particularly in geometric group theory. One of the most studied families of Artin groups--and historically the first to be studied--is that of spherical-type Artin groups. These groups served as a model for the development of the theory of Garside groups \cite{DehornoyParis}, and braid groups are iconic examples of them \cite{Art47}.

Let $\Gamma$ be a Coxeter graph that can be written as a disjoint union of full subgraphs $\Gamma=\Gamma_1\sqcup\cdots \sqcup \Gamma_k$. It is easily seen that the associated Artin and Coxeter groups naturally decompose as $\A[\Gamma]=\A[\Gamma_1]\times \cdots \times \A[\Gamma_k]$ and $\W[\Gamma]=\W[\Gamma_1]\times \cdots \times \W[\Gamma_k]$, respectively. If $\Gamma$ is connected, then $\Gamma$, $\W[\Gamma]$ and $\A[\Gamma]$ are called \textit{irreducible}, while they are called \textit{reducible} otherwise. The Coxeter graphs that give rise to irreducible Coxeter groups of spherical type have been completely classified by Coxeter \cite{Cox34,Cox35}, and they must belong to the list of graphs illustrated in Figure \ref{fig:Cox-graphs-spherical-type}.

\begin{figure}
   \centering
\includegraphics[width=0.75\linewidth]{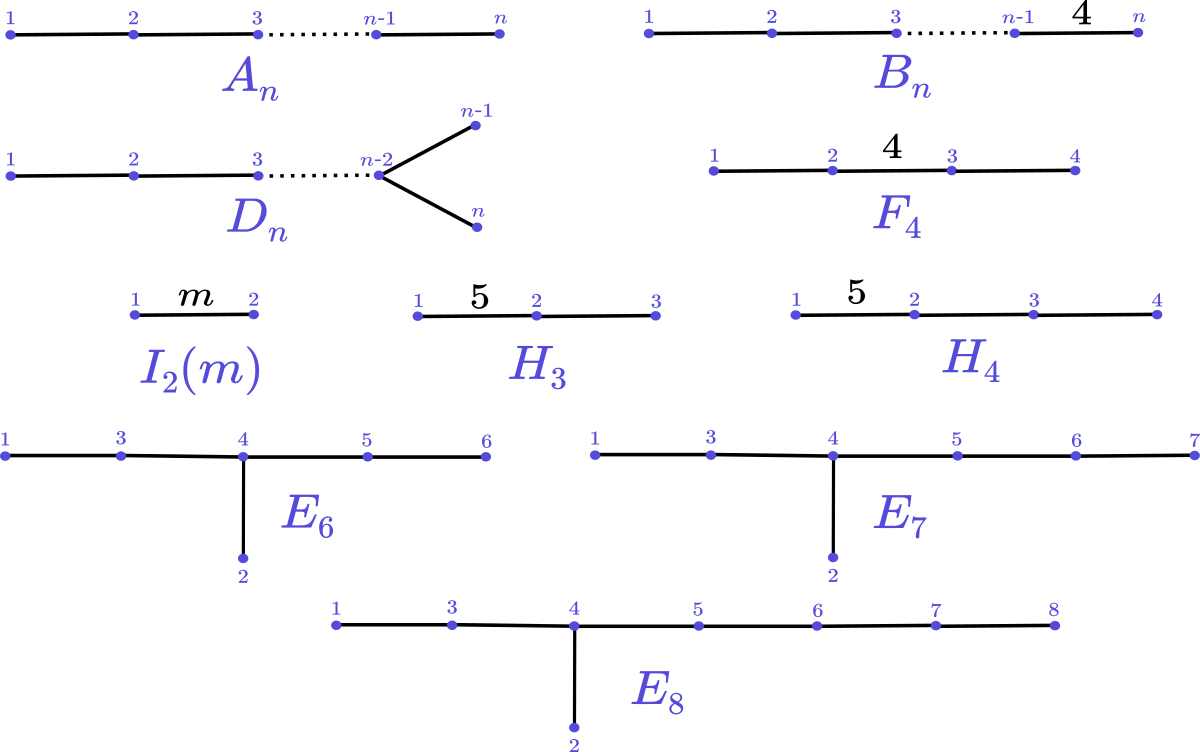}
    \caption{The irreducible Coxeter graphs of spherical type.}
    \label{fig:Cox-graphs-spherical-type}
\end{figure}

A classical question in the field, known as the \emph{isomorphism problem}, consists of determining when two Coxeter graphs define isomorphic Artin groups. There are examples of distinct Coxeter graphs that define isomorphic Artin groups (see \cite{BMMN02}). On the other hand, it was shown in \cite{Par04} that, if two spherical-type Coxeter graphs define isomorphic Artin groups, then the corresponding Coxeter graphs are isomorphic. Then the class of spherical-type Artin groups is said to be \emph{rigid}. The isomorphism problem remains widely open; beyond spherical-type Artin groups, it has been solved only for right-angled Artin groups \cite{Dro87} and for large-type Artin groups \cite{Vas23}.

In this work we tackle a stronger form of rigidity for spherical type Artin groups: the \textit{profinite rigidity}. Given a finitely generated group $G$, denote by $\F(G)$ the set of isomorphism classes of finite quotients of $G$. The group $G$ is called \textit{profinitely rigid relative to a class of groups} $\C$ if $G\in \C$ and, for any group $H\in \C$, whenever $\F(G)=\F(H)$, then $G\cong H$. Recall that a finitely generated group $G$ is called \textit{residually finite} if, for all $g\in G\backslash \{\id\}$, there exist a finite group $Q$ and a surjective homomorphism $p:G\longrightarrow Q$ such that $p(g)\neq \id$. A finitely generated residually finite group $G$ is called
\textit{profinitely rigid} if $G$ is profinitely rigid among all finitely generated residually finite groups.

In general, studying profinite properties of groups is only interesting when the groups are residually finite (more details are given in Section \ref{section-preliminaries}). Artin groups of spherical type were shown to be linear (see \cite{Digne} and \cite{CohenWales}), thus, by a classical result of Mal'cev \cite{Malcev1940}, they are residually finite. The main result of this article is the following.

\begin{thm}\label{mainthm}
Let $\Gamma$ and $\Omega$ be two Coxeter graphs of spherical type.
If $\Gamma$ is irreducible and $\widehat{\A [\Gamma]} \cong \widehat{\A [\Omega]}$, then $\Gamma \cong \Omega$.
\end{thm}

The proof of Theorem \ref{mainthm} is independent of \cite{Par04}.
In particular, it yields a new proof of the main theorem of \cite{Par04} in the case where one of the Coxeter graphs is irreducible.

\begin{corl}[Paris \cite{Par04}]\label{mainthmC1}
Let $\Gamma$ and $\Omega$ be two Coxeter graphs of spherical type.
If $\Gamma$ is irreducible and $\A[\Gamma] \cong \A[\Omega]$, then $\Gamma \cong \Omega$.
\end{corl}

Another direct consequence of Theorem \ref{mainthm} is the following.

\begin{corl}\label{mainthmC2}
    Irreducible Artin groups of spherical type are profinitely rigid in the class of Artin groups of spherical type.
\end{corl}

Theorem \ref{mainthm} will follow from a deep analysis of the finite quotients of irreducible Artin groups of spherical type. From now on, when referring to a quotient $Q$ of some group $G$, we will say that $Q$ is the \textit{smallest non-abelian quotient of} $G$ (sometimes shortened in \textit{SNAQ}) if $Q$ is finite and, for every finite non-abelian quotient $R$ of $G$, we have either $|R|> |Q|$, or $R\cong Q$. Note that such a quotient may not exist.

For the braid groups, Kolay \cite[Theorem 1]{Kolay} showed the following result, previously conjectured by Margalit (see \cite{CKLP,ScherichVerberne_2023}):

\begin{thm}[Kolay \cite{Kolay}]\label{thm-Kolay}\leavevmode
\begin{enumerate}
    \item [(i)]Let $n\in\mathbb{N}_{\geq 5}$. The smallest non-abelian quotient of the braid group $\B_{n}$ is the symmetric group $\S_n$.
    \item [(ii)] The symmetric group on three elements $\S_3$ is the smallest non-abelian quotient of $\B_3$ and $\B_4$.
\end{enumerate}
\end{thm}

Observe that $\B_2$ is isomorphic to $\Z$, thus it cannot have non-abelian quotients. Recall that $\B_n=\A[A_{n-1}]$ and $\S_n=\W[A_{n-1}]$. Inspired by Kolay's theorem, we show the following, which is another crucial result in this work.

\begin{thm}\label{thm-SNAQ}
Let $\Gamma$ be an irreducible Coxeter graph of spherical type.
\begin{enumerate}
    \item [(i)] Let $n\geq 5$. If $\Gamma\in \{A_{n-1},B_n,D_n\}$, then the smallest non-abelian quotient of $\A[\Gamma]$ is $\W[A_{n-1}]=\S_n$.
    \item [(ii)] If $\Gamma\in\{A_2,A_3,B_2,B_3,B_4,D_4,F_4\}$ or $\Gamma=I_2(m)$ with $4\mid m$, then the smallest non-abelian quotient of $\A[\Gamma]$ is $\W[A_2]=\S_3$.
    \item[(iii)] If $\Gamma=I_2(m)$ and ${4\nmid m}$, then the smallest non-abelian quotient of $\A[\Gamma]$ is $\W[I_2(p)]$, where $p$ is the smallest odd divisor of $m$.
    \item [(iv)] If $\Gamma\in\{E_6,E_7,E_8,H_3,H_4\}$, then the smallest non-abelian quotient of $\A[\Gamma]$ is $\W[\Gamma]/Z(\W[\Gamma])$.
\end{enumerate}
\end{thm}

\begin{rem}
Let $\Gamma$ be a not necessarily irreducible Coxeter graph of spherical type. Assume that $\Gamma$ is not a disjoint union of isolated vertices, that is, $\A[\Gamma]$ is non-abelian. By Theorem \ref{thm-SNAQ}, the Artin group $\A[\Gamma]$ admits a smallest non-abelian quotient, which can be determined as follows. Let $\Gamma_1, \dots, \Gamma_k$ be the connected components of $\Gamma$ that are not isomorphic to $A_1$. For each $1 \le i \le k$, let $Q_i$ denote the smallest non-abelian quotient of $\A[\Gamma_i]$. After reordering the indices if necessary, we may assume that $|Q_1| \le |Q_2| \le \cdots \le |Q_k|$. It follows readily from Theorem \ref{thm-SNAQ} that, if $Q$ and $Q'$ are the smallest non-abelian quotients of two irreducible spherical-type Artin groups such that $|Q| = |Q'|$, then $Q \cong Q'$. It follows that $Q_1$ is the smallest non-abelian quotient of $\A[\Gamma]$.
\end{rem}

Our proof strategy is as follows. We use Theorem \ref{thm-Kolay} along with many case-by-case analysis and computer calculations to prove Theorem \ref{thm-SNAQ}. Then, along with other new intermediate results, we use Theorem \ref{thm-SNAQ} to deduce Theorem \ref{mainthm}. 

The article is organized as follows. In Section \ref{section-preliminaries}, we give preliminaries on profinite completions. In Section \ref{section-SNAQ}, we provide a proof of Theorem \ref{thm-SNAQ} for all types but $E$ or $H$, which are handled separately in Section \ref{sec-SNAQ-type-E-H}. Then, in Section \ref{subs-intermediate-results}, we prove Theorem \ref{mainthm} by providing many intermediate results, and assuming those proven in Section \ref{sec-SNAQ-type-E-H}.

We mention that Sam Hughes, Thomas Ng, Kaitlin Ragosta, Nancy Scherich, and Yvon Verberne also have independently obtained similar results concerning the smallest quotients of spherical-type Artin groups, which appear in the arXiv preprint \cite{HNRSV}. However, we became aware of their work only after completing this article.

\begin{acknow}    
The authors thank Ivan Marin for enlightening discussions about profinite completions of Artin groups. The second author wishes to thank the Institut de Mathématiques de Bourgogne for the two invitations that allowed part of this work to be done. The first author also thanks the Institut de Mathématiques de Bourgogne for the invitation in February 2026. The first author is supported by the EPSRC Standard Research Grant UKRI1018. The second author is partially supported by the French project ``CORTIPOM'' (ANR-21-CE40-0019) of the ANR. The third author is partially supported by the French project ``CaGeT'' (ANR-25-CE40-4162) of the ANR. The Institut de Mathématiques de Bourgogne (IMB) receives support from the EIPHI Graduate School (grant ANR-17-EURE-0002).
\end{acknow}

\begin{discloAI}
    In this work, the use of LLMs (namely the free versions of ChatGPT and Gemini) is confined to minor corrections and suggestions for English language, as well as debugging some GAP programs. After using these tools, the authors reviewed and edited the content as needed. The results and proofs come from the authors, which take full responsibility for their correctness.
\end{discloAI} 

\section{Preliminaries on profinite completions of groups}\label{section-preliminaries}

\begin{defn}
Let $G$ be a finitely generated group (not necessarily residually finite for the moment), and let $\N$ denote the collection of finite index normal subgroups of $G$. Since $G\in \N$, this set in non-empty, and it can be made into a directed set (i.e. a poset in which every finite subset has an upper bound) by imposing for all $M,N\in \N$
\[
M\leq N \quad \text{whenever }\quad M\supseteq N\,.
\]
With this notion, there are natural surjective homomorphisms $\varphi_{N,M}:G/N\longrightarrow G/M$, which turn $(G/N, \varphi_{N,M}, \N)$ into an inverse system.
The inverse limit of  $(G/N, \varphi_{N,M}, \N)$ is denoted by $\widehat{G}$ and is called the \textit{profinite completion} of $G$.
\[
\widehat{G}=\underset{\underset{N\in \N}{\longleftarrow} }{\lim}\;(G/N)=\left\{(gN)\in \prod_{N\in \N}G/N\; \bigm\lvert\; \varphi_{N,M}(gN)=gM\; \text{ whenever } \; M\leq N\right\}.
\]
\end{defn}

Standard references for this topic are \cite{Reid-Survey} and \cite{RibesZalesskii2010}. 

There is a natural group homomorphism  
\[
i:G\longrightarrow \widehat{G} \subseteq \prod_{N\in \N}G/N,\qquad g\longmapsto (gN)_{N\in \N}\,.
\]
Observe that such a map is injective if and only if the group $G$ is residually finite. For this reason, from now on, we will always assume this hypothesis.

The profinite completion of a group is naturally equipped with a topology: for each $N\in \N$, first equip each $G/N$ with the discrete topology, then $\prod_{N\in \N}G/N$ is a compact space and $\widehat{G}$ can be identified with $\overline{i(G)}$ in $\prod_{N\in \N}G/N$. From now on, given two groups $G$ and $H$, when we write $\widehat{G}\cong\widehat{H}$ we will always mean that they are isomorphic as topological groups, with the topology defined above.

\begin{defn}
    Let $G$ be a group. Define $\F(G)$ to be the set of isomorphism classes of finite quotients of $G$.  
\end{defn}

It turns out that the data of $\F(G)$ suffices to recover the profinite completion of $G$. In particular, Dixon, Formanek, Poland and Ribes show the following.

\begin{thm}[Dixon--Formanek--Poland--Ribes \cite{DFPR}]\label{thm-same-finite-quotients-same-profinite-completion}
    Let $G,H$ be two finitely generated groups. The groups $\widehat G$ and $\widehat H$ are isomorphic as topological groups if and only if $\mathcal F(G)=\mathcal F(H)$. 
\end{thm}

This result holds in the full generality of finitely generated group, but we will use it in the restricted context of residually finite groups. 
An immediate corollary which we shall use extensively throughout the article is the following.

\begin{corl}\label{cor-smallest-diff-profinite-diff}
    If two finitely generated residually finite groups have non-isomorphic smallest non-abelian quotients, then their profinite completions are not isomorphic. 
\end{corl}

We now state some useful results about profinite completions, whose proofs can be found in \cite{Reid-Survey}. Given a group $G$, we will denote by $G^\ab:=G/[G,G]$ its abelianized. The following proposition can be found in \cite[Proposition 3.2]{Reid-Survey}.

\begin{prop}\label{prop-profinite-same-abelianization}
   If $G,H$ are finitely generated residually finite groups such that $\widehat{G}\cong \widehat{H}$, then $G^\ab\cong H^\ab$. 
\end{prop}

Given two groups $G$ and $P$, denote by $\mathrm{Hom}(G,P)$ the set of homomorphisms $G\to P$ and by $\mathrm{Epi}(G,P)$ the set of epimorphisms $G\to P$. The following proposition can be found in \cite[Lemma 4.1 and Corollary 4.2]{Reid-Survey}.

\begin{prop}\label{prop-profinite-same-epis}
    Given two finitely generated residually finite groups $G$ and $H$, if $\widehat G\cong\widehat H$, then for any finite group $P$ we have $|\mathrm{Hom}(G,P)|=|\mathrm{Hom}(H,P)|$ and $|\mathrm{Epi}(G,P)|=|\mathrm{Epi}(H,P)|$.
\end{prop}

A consequence of Proposition \ref{prop-profinite-same-epis} which shall be useful to us is the following (see \cite[Lemma 4.10]{Reid-Survey} for the proof).

\begin{lem}\label{lem-strict-quotient-rigidity}
    Given two finitely generated residually finite groups $G$ and $H$, if $H$ is a strict quotient of $G$, then $\widehat G$ and $\widehat H$ are not isomorphic. 
\end{lem}

Now, a useful result, which is specific to profinite completions of Artin groups of spherical type, is the following.

\begin{thm}[Marin \cite{Marin}]\label{prop-rigidity-one-class-marin}
Let $\Gamma,\Omega$ be two irreducible Coxeter graphs of spherical type which are not of type $E,F$ or $H$ and such that $(\A[\Gamma])^{\mathrm{ab}}=(\A[\Omega])^{\mathrm{ab}}=\mathbb Z$. If $\widehat{\A[\Gamma]}\cong \widehat{\A[\Omega]}$, then $\Gamma=\Omega$.
\end{thm}

We now briefly discuss some useful results about the cohomology of Artin groups of spherical type and that of their profinite completions. For an introduction to the cohomology of groups and profinite groups, see \cite[Sections 6 and 7]{RibesZalesskii2010}. Note that rather than diving in the details of group cohomology, we use these results as a toolkit. 
A useful notion relating the cohomology of a group $G$ with its profinite completion is that of good groups, as introduced by Serre in \cite{Serre}.

\begin{defn}
A group $G$ is \textit{good in the sense of Serre} if for every finite $\widehat{G}-$module $M$ and all $i\in\mathbb N$, the map
\begin{equation*}
        H^i(\widehat{G},M)\to H^i(G,M)
    \end{equation*}
is an isomorphism.
\end{defn}

The interest of this notion to us is made clear by the following results.

\begin{thm}[Marin \cite{Marin}]\label{prop-not-E-F-H-is-good}
  If $\Gamma$ is an irreducible Coxeter graph of spherical type which is not of type $E,F$ or $H$, then $\A[\Gamma]$ is good in the sense of Serre. 
\end{thm} 

Let $\Gamma$ be a Coxeter graph none of whose connected components is of type $E$, $F$, or $H$. Now, using Theorem \ref{prop-not-E-F-H-is-good}, we turn to show that the profinite cohomological dimension of $\A[\Gamma]$ is equal to its rank (see Corollary \ref{cor-rank-profinite}). This invariant will be useful later to distinguish certain profinite completions of Artin groups. The first ingredient in the proof of Corollary \ref{cor-rank-profinite} is the following proposition. It is well known to experts, and a proof can be found in \cite[Proposition 3.1]{Par04}.

\begin{prop}\label{prop-cohomological-dimension-artin}
    The cohomological dimension of an Artin group of spherical type is equal to its rank. 
\end{prop}

The second ingredient is the following.

\begin{prop}\label{prop-rank-profinite}
    Given $\Gamma$ a Coxeter graph of spherical type of rank $n$, there exists a finite $\A [\Gamma]$-module $M$ such that $H^{n}(\A[\Gamma],M)\neq \{0\}$. 
\end{prop}

\begin{proof}
To simplify the notation, we write $\W = \W [\Gamma]$ and $\A =\A [\Gamma]$ whenever there is no ambiguity regarding the choice of the Coxeter graph $\Gamma$.
Let $S$ be the vertex set of $\Gamma$, viewed also as a generating set for both $\W$ and $\A$.
We denote by $\theta : \A \to \W$ the canonical epimorphism that maps $s$ to $s$ for every $s \in S$.
This epimorphism admits a natural set-section $\tau : \W \to \A$, defined as follows.
Let $w \in \W$.
Denote by $\ell (w)$ its word length.
Recall that an expression $w = s_1 s_2 \cdots s_k$ over $S$ is called \emph{reduced} if $k = \ell (w)$.
Choose a reduced expression $w= s_1 s_2 \cdots s_k$, and set $\tau (w) = s_1 s_2 \cdots s_k$, viewed as an element of $\A$.
It follows from \cite{Matsumoto} that the definition of $\tau (w)$ does not depend on the choice of the reduced expression.

For $J \subseteq S$, let $\W_J$ denote the subgroup of $\W$ generated by $J$.
For $J' \subseteq J \subseteq S$, define $\W_J^{J'} = \{ w \in \W_J \mid \ell (ws) > \ell (w) \text{ for all } s \in J' \}$, and denote by $\rho_J^{J'}$ the element of the group algebra $\mathbb Z [\A]$ given by
\[
\rho_J^{J'} = \sum_{w \in \W_J^{J'}} (-1)^{\ell (w)} \tau(w)\,.
\]

Let $M$ be an $\A$-module.
For each $0 \le k \le n = |S|$, let $\mathcal P_k (S)$ denote the set of subsets of $S$ of cardinality $k$, and let $C^k (M) = M^{\mathcal P_k (S)}$ be the module of maps from $\mathcal P_k (S)$ to $M$.
Fix an ordering $S = \{ s_1 < s_2 < \cdots < s_n \}$ and define a homomorphism $\delta^{n-1} : C^{n-1} (M) \to C^n (M)$ as follows.
For $f \in C^{n-1} (M)$, set
\[
\delta^{n-1}(f)(S) = \sum_{j=1}^n (-1)^j \rho_S^{S \setminus \{s_j\}} \cdot f(S \setminus \{s_j\})\,.
\]
Then, by \cite{ConciniSalvetti2},
\[
H^n (\A, M) \cong C^n (M) / \mathrm{im} (\delta^{n-1})\,.
\]

Since $C^n (M) \cong M$, to show that $H^n (\A, M) \neq \{0\}$ (assuming, of course, that $M \neq \{ 0 \}$), it suffices to show that $\mathrm{im} (\delta^{n-1}) = \{0\}$.
By the definition of $\delta^{n-1}$, it suffices to show that $\rho_S^{S \setminus \{s_j\}} \cdot M = \{0\}$ for every $1 \le j \le n$.
Let $p$ be a prime number, and let $\mathbb F_p$ denote the finite field with $p$ elements.
We shall construct such a module $M$, where $M$ is an $\mathbb F_p$-vector space.

First, assume that $\Gamma$ is irreducible, that is, connected.
Let $P \in \mathbb F_p [X]$ be a non-constant monic polynomial such that $X$ is invertible in $\mathbb F_p [X]/ (P)$.
We define an $\mathbb{F}_p$-linear action of $\A$ on $\mathbb{F}_p[X]/(P)$ by letting each standard generator act via multiplication by $X$.
Then we set $M = \mathbb F_p [X]/ (P)$, endowed with this $\A$-module structure.

For every $w \in \W$, the element $\tau(w)$ acts on $M$ by multiplication by $X^{\ell(w)}$.
Hence, $\rho_S^{S \setminus \{s_j\}}$ acts on $M$ by multiplication by the polynomial $Q_j = \sum_{w \in \W_S^{S \setminus \{s_j\}}} (-X)^{\ell(w)}$.
It follows from \cite[Chapitre 4, Exercise 1.26]{Bourbaki} that $Q_j = \W_S (-X)/\W_{S \setminus \{s_j\}} (-X)$, where, for each $J \subseteq S$, $\W_J (X) = \sum_{w \in \W_J} X^{\ell(w)}$.
Moreover, it is shown in \cite{Solomon} that
\[
\W_S(X) = \prod_{i=1}^n \frac{X^{m_i+1} -1}{X-1}\,,
\]
where $m_1 \le m_2 \le \cdots \le m_n$ are the exponents of $\W$.
We know that $m_n+1 = h_{\W}$ is the Coxeter number of $\W$, and we easily deduce from \cite[Table 2, Page 80]{Humphreys} that $h_{\W_J} < h_{\W}$ for every $J \subsetneq S$.
Note that the hypothesis that $\Gamma$ is irreducible is required for this last inequality.

Now, choose $p$ which does not divide $h_{\W}$.
Choose an irreducible factor $\Psi$ in $\mathbb F_p [X]$ of the reduction modulo $p$ of the cyclotomic polynomial $\Phi_{h_{\W}}$.
The polynomial $\Psi$ is not necessarily the reduction modulo $p$ of $\Phi_{h_{\W}}$, but all of its roots in $\overline{\mathbb F_p}$ are primitive $h_{\W}$-th roots of unity.
In particular, $0$ is not a root of $\Psi$.
Now, we set $P = \Psi (-X)$.
Note that $X$ is invertible in $\mathbb F_p [X] / (P)$, since $0$ is not a root of $P$.

Let $1 \le j \le n$.
To prove that $\rho_S^{S \setminus \{s_j\}} \cdot M = \{0\}$, it suffices to show that $Q_j = 0$ in $\mathbb F_p [X]/ (P)$, that is, that $P$ divides $Q_j$ in $\mathbb F_p [X]$.
Let $\mu \in \overline{\mathbb F_p}$ be a root of $\Psi$.
We know that $\W_{S \setminus \{s_j\}} (X) = \prod_{i=1}^{n-1} \frac{X^{m_i'+1}-1}{X-1}$, where $m_1' \le m_2' \le \cdots \le m_{n-1}'$ are the exponents of $\W_{S \setminus \{s_j\}}$.
Moreover, $m_{n-1}'+1 = h_{\W_{S \setminus \{s_j\}}} < h_{\W}$, and $\mu$ is a primitive $h_{\W}$-th root of unity.
Therefore, $\mu$ is not a root of $\W_{S \setminus \{s_j\}}(X)$, that is, $\W_{S \setminus \{s_j\}} (\mu) \neq 0$ in $\overline{\mathbb F_p}$.
On the other hand, since $\frac{X^{h_{\W}}-1}{X-1}$ divides $\W_S (X)$, we have $\W_S (\mu) = 0$ in $\overline{\mathbb F_p}$.
Since $Q_j (X) = \W_S (-X) / \W_{S \setminus \{s_j\}} (-X)$, it follows that $-\mu$ is a root of $Q_j$.
Hence, every root of $\Psi(-X)$ is a root of $Q_j$.
Since $P=\Psi(-X)$, we conclude that $P$ divides $Q_j$ in $\mathbb{F}_p[X]$.

Now, assume that $\Gamma$ is not connected.
Let $\Gamma_1, \dots, \Gamma_r$ be the connected components of $\Gamma$.
Choose a prime number $p$ which does not divide $h_{\W [\Gamma_i]}$ for any $i \in \{1, \dots, r\}$.
For each $i \in \{1, \dots, r\}$, let $M_i$ be the $\A[\Gamma_i]$-module constructed above, and define
\[
M = M_1 \otimes_{\mathbb F_p} M_2 \otimes_{\mathbb F_p} \cdots \otimes_{\mathbb F_p} M_r\,.
\]
We equip $M$ with the natural action of $\A [\Gamma] = \A [\Gamma_1] \times \cdots \times \A[\Gamma_r]$ on $M$ defined by
\[
(a_1, a_2, \dots, a_r) \cdot (x_1 \otimes x_2 \otimes \cdots \otimes x_r) =
(a_1 \cdot x_1) \otimes (a_2 \cdot x_2) \otimes \cdots \otimes (a_r \cdot x_r)\,.
\]
For each $1 \le i \le r$, let $S_i$ denote the vertex set of $\Gamma_i$.
Let $s \in S$.
Without loss of generality we may assume that $s \in S_1$.
Then $\W_S^{S \setminus \{s\}} = \W_{S_1}^{S_1 \setminus \{s\}}$, hence $\rho_S^{S \setminus \{s\}} = \rho_{S_1}^{S_1 \setminus \{s\}} \in \mathbb Z [\A [\Gamma_1]]$. 
It follows that
\[
\rho_S^{S \setminus \{s\}} \cdot M =
(\rho_{S_1}^{S_1 \setminus \{s\}} \cdot M_1) \otimes_{\mathbb F_p} M_2 \otimes_{\mathbb F_p} \cdots \otimes_{\mathbb F_p} M_r =
\{ 0 \} \otimes_{\mathbb F_p} M_2 \otimes_{\mathbb F_p} \cdots \otimes_{\mathbb F_p} M_r =
\{0\}\,.
\]
This completes the proof of Proposition~\ref{prop-rank-profinite}.
\end{proof}

\begin{corl}\label{cor-rank-profinite}
Let $\Gamma$ be a Coxeter graph of spherical type with no connected component of type $E$, $F$, or $H$.
Then the profinite cohomological dimension of $\widehat{\A [\Gamma]}$ is equal to the rank of $\Gamma$.
\end{corl}

\begin{proof}
By \cite[Section 7]{RibesZalesskii2010}, if a group $G$ is good in the sense of Serre and has finite cohomological dimension $n \in \mathbb{N}$, and there exists a finite $G$-module $M$ such that $H^n(G, M) \neq \{0\}$, then the profinite cohomological dimension of $\widehat{G}$ is $n$.

Let $n$ be the rank of $\Gamma$.
By Proposition \ref{prop-cohomological-dimension-artin}, the cohomological dimension of $\A [\Gamma]$ is $n$.
Let $\Gamma_1, \dots, \Gamma_r$ be the connected components of $\Gamma$.
By Theorem \ref{prop-not-E-F-H-is-good}, each group $\A [\Gamma_i]$ is good in the sense of Serre, and, by \cite[Proposition 3.4]{GrJaZa08}, a finite direct product of good groups in the sense of Serre is a good group in the sense of Serre.
It follows that $\A [\Gamma]$ is good in the sense of Serre.
Finally, by Proposition \ref{prop-rank-profinite}, there exists a finite $\A [\Gamma]$-module $M$ such that $H^n (\A [\Gamma], M) \neq \{0\}$.
We conclude that the profinite cohomological dimension of $\widehat{\A [\Gamma]}$ is $n$.
\end{proof}

\section{Smallest non-abelian quotients}\label{section-SNAQ}

The aim of this section is to determine the smallest non-abelian quotients of irreducible Artin groups of spherical type, except for those of type $E$ or $H$. These two cases are way more technical and require computations with GAP (\cite{GAP4}), hence we will treat them in Section \ref{sec-SNAQ-type-E-H}.

\begin{rem}\label{rem-generators-odd-conjugate}
    If two generators $a,b$ in $S$ are joined in $\Gamma$ by an odd-labelled edge, then $a$ and $b$ are conjugate in $\W[\Gamma]$ and in $\A[\Gamma]$. Consequently, these generators are identified in the abelianized $(\A[\Gamma])^{\mathrm{ab}}$. We see in particular that the abelianized of an irreducible Artin group of spherical type is either $\Z$ or $\Z^2$.
\end{rem}
We start by making a general observation which will be useful in the search of $\mathrm{SNAQ}(\A[\Gamma])$.

\begin{lem}\label{lem-SNAQ-is-centreless}
    Let $\Gamma$ be a Coxeter graph such that $\A[\Gamma]^{\mathrm{ab}}\cong\mathbb Z$ and let $P$ be a non-abelian quotient of $\A[\Gamma]$.
    \begin{enumerate}
        \item[(i)] The group $P/Z(P)$ is not abelian.
        \item[(ii)] If $P$ has no non-abelian quotient, then $P$ has trivial centre.
    \end{enumerate}
\end{lem}

\begin{proof}
If $P/Z(P)$ is abelian, then it is a quotient of $\A[\Gamma]^{\mathrm{ab}}\cong\Z$, hence it is cyclic. By a classical result, this implies that $P$ itself is abelian, which is a contradiction. This shows (i). Point (ii) follows directly from Point (i). This concludes the proof. 
\end{proof}

The next two lemmas are used in the proof of Proposition \ref{prop-SNAQ-firstcases}, and their proofs are left as an exercise to the reader.

\begin{lem}\label{lem-S3-smallest-nonabelian}
    The symmetric group $\S_3$ is the the only non-abelian group of order 6 and there is no non-abelian group of order strictly smaller that 6.
\end{lem}

\begin{lem}\label{lem-only-group-order2p}
    For any prime $p$, the group $\W[I_2(p)]$ is the unique non-abelian group of order $2p$, up to isomorphism. 
\end{lem}

We now prove the following result.

\begin{prop}\label{prop-SNAQ-firstcases}
    Points (i)--(iii) of Theorem \ref{thm-SNAQ} hold.
\end{prop}

\begin{proof}
\underline{$\Gamma=A_n$, $n\geq 4$:} This is the content of Theorem \ref{thm-Kolay}.

\underline{$\Gamma=B_n$, $n\geq 5$:} First, observe that the group $\W[A_{n-1}]$ is a quotient of $\A[B_n]$. If we denote by $a_1,\ldots,a_{n-1}$ the standard generators of $\A[A_{n-1}]$ and by $\sigma_1,\ldots,\sigma_n$ the standard generators of $\A[B_{n}]$, then the map $\A[B_n]\to \A[A_{n-1}]$ sending $\sigma_i$ to $a_i$ for all $i=1,\ldots,n-1$ and $\sigma_n$ to $\id$ is a surjective homomorphism. By Theorem \ref{thm-Kolay}, $\S_n$ is a quotient of $\A[A_{n-1}]$, and thus of $\A[B_n]$. We know show that there is no smaller non-abelian quotient.

Let $P$ be a non-abelian quotient of $\A[B_n]$ and $\varphi:\A[B_n]\to P$ an epimorphism. Observe that $\A[B_n]/\llangle \sigma_1\sigma_2^{-1}\rrangle\cong \mathbb Z^2$. Indeed, if two generators $\sigma_i, \sigma_j$ satisfy both the commutation and the braid relations, they must coincide. Therefore, imposing the relation $\sigma_1=\sigma_2$ produces $\sigma_1=\sigma_2=\cdots=\sigma_{n-1}$, while $\sigma_{n}$ remains distinct and commutes with $\sigma_{n-1}$. Also observe that the subgroup of $\A[B_n]$ generated by $\sigma_1,\ldots,\sigma_{n-1}$ is isomorphic to $\A[A_{n-1}]=\B_n$. Now, since $\A[B_n]/\llangle \sigma_1\sigma_2^{-1}\rrangle\cong \mathbb Z^2$ and $P$ is assumed non-abelian, we have $\varphi (\sigma_1) \neq \varphi (\sigma_2)$, hence $\varphi(\langle \sigma_1,\dots,\sigma_{n-1}\rangle)<P$ is not abelian, and is a quotient of $\B_n$. Thus, $|P|\geq |\varphi(\langle \sigma_1,\dots,\sigma_{n-1}\rangle)|\geq n!$ by Theorem \ref{thm-Kolay}. Moreover, again by Theorem \ref{thm-Kolay}, if $|P|=n!$, then $\varphi(\langle\sigma_1,\dots,\sigma_{n-1}\rangle)=P=\W[A_{n-1}]=\S_n$.

\underline{$\Gamma=D_n$, $n\geq 5$:} First, observe that the group $\W[A_{n-1}]$ is a quotient of $\A[D_n]$.
As in the previous case, we denote by $a_1, \dots, a_{n-1}$ the standard generators of $A [A_{n-1}]$ and by $\sigma_1, \dots, \sigma_n$ the standard generators of $\A [D_n]$.
Then the map $\A [D_n] \to \A [A_{n-1}]$ sending $\sigma_i$ to $a_i$ for all $i=1, \dots, n-1$ and $\sigma_n$ to $a_{n-1}$ is a surjective homomorphism.
Since $\W [A_{n-1}]=\S_n$ is a quotient of $\A [ A_{n-1}]$, it follows that $\W [A_{n-1}]$ is a quotient of $\A [D_n]$. Again, we show that this quotient is the smallest. Let then $P$ be a non-abelian quotient of $\A[D_n]$ and $\varphi:\A[D_n]\to P$ the quotient projection. Similarly to the previous paragraph, observe that $\A [D_n] / \llangle \sigma_1 \sigma_2^{-1} \rrangle \cong \A[D_n]^{\mathrm{ab}}\cong\mathbb Z$, and the subgroup of $\A[D_n]$ generated by $\sigma_1,\ldots,\sigma_n$ is isomorphic to the braid group on $n$ strands $\B_n$. 
Since $\A [D_n] / \llangle \sigma_1 \sigma_2^{-1} \rrangle \cong \mathbb Z$ and $P$ is assumed to be non-abelian, we have $\varphi (\sigma_1) \neq \varphi (\sigma_2)$, hence $\varphi (\langle \sigma_1, \dots, \sigma_{n-1} \rangle )<P$ is a non-abelian quotient of $\A[A_{n-1}]=\B_n$.
Thus, using Theorem \ref{thm-Kolay}, we obtain $|P|\geq n!$. Moreover, again by Theorem \ref{thm-Kolay}, if $|P|=n!$, then $\varphi(\langle\sigma_1,\dots,\sigma_{n-1})\rangle)\cong\W[A_{n-1}]=\S_n$, and therefore $P\cong\S_n$.

\underline{$\Gamma\in \{A_2,A_3,B_2,B_3,B_4,D_4,F_4\}\cup\{I_2(4m) \mid m \in \mathbb N_{\ge 2}\}$:} 
By Lemma \ref{lem-S3-smallest-nonabelian}, $\W [A_2] = \S_3$ is the unique non-abelian group of order at most $6$, hence it suffices to show that $\W [A_2]$ is a quotient of $\A [\Gamma]$.
The cases $\Gamma = A_2$ and $\Gamma = A_3$ are treated in Theorem \ref{thm-Kolay}.
As in the case $\Gamma = B_n$ with $n \ge 5$, we can show that $\A [A_{n-1}]$ is a quotient of $\A [B_n]$ for $n=3,4$, hence $\W [A_2]$ is a quotient of $\A [B_n]$ for $n=3,4$.
Similarly, the group $\W[A_2]$ is a quotient of $\A [D_4]$.
Let $a_1, a_2$ be the standard generators of $\A [A_2]$, and let $f_1, f_2, f_3, f_4$ be the standard generators of $\A [F_4]$, numbered as in Figure \ref{fig:Cox-graphs-spherical-type}.
There exists a surjective homomorphism $\A [F_4] \to \A [A_2]$ that maps $f_i$ to $a_i$ for $i=1,2$, and $f_i$ to $\mathrm{id}$ for $i=3,4$, hence $\W [A_2]$ is a quotient of $\A [F_4]$.
It remains to consider the case where $\Gamma \in \{ B_2\} \cup \{ I_2 (4m) \mid m \in \mathbb N_{\ge 2}\}$.

Set $B_2 = I_2(4)$, and take $\Gamma \in \{ I_2(4m) \mid m \in \mathbb{N}_{\ge 1} \}$.
Let $s_1, s_2$ be the standard generators of $\A[\Gamma]$, and let $a_1, a_2$ be the standard generators of $\W[A_2]$.
Set $u = s_1 s_2$.
It can be verified that $\A[\Gamma]$ admits the presentation $\A[\Gamma] = \langle u, s_1 \mid s_1 u^{2m} = u^{2m} s_1 \rangle$.
It follows that there exists a surjective homomorphism $\A[\Gamma] \to \W[A_2]$ that maps $u$ to $a_1$ and $s_1$ to $a_2$, and hence $\S_3$ is a quotient of $\A[\Gamma]$.

\underline{$\Gamma=I_2(m)$, $m$ odd:} Let $p$ be the smallest prime divisor of $m$. First, the group $\W[I_2(p)]$ is a quotient of $\A[\Gamma]$ since $\A[I_2(p)]$ is. Let then $P$ be a non-abelian quotient of $\A[\Gamma]$ and $\varphi:\A[\Gamma]\to P$ be an epimorphism. By Lemma \ref{lem-SNAQ-is-centreless}, we can assume that $P$ is centreless, so that $P$ is a quotient of $\A[\Gamma]/Z(\A[\Gamma])$, that is, $\varphi$ induces an epimorphism $\overline{\varphi} : \A [\Gamma]/ Z (\A [\Gamma]) \to P$.
Let $s_1, s_2$ be the standard generators of $\A [\Gamma]$.
Set $u =s_1 s_2$ and $\Delta = (s_1 s_2)^k s_1 = (s_2 s_1)^k s_2$, where $m = 2k+1$.
It can be verified that $\A [\Gamma]$ admits the presentation $\A [\Gamma] = \langle u, \Delta \mid u^m=\Delta^2 \rangle$ and $Z (\A [\Gamma])$ is the cyclic subgroup generated by $u^m = \Delta^2$.
This implies that $\A [\Gamma] / Z (\A [\Gamma]) \cong \mathbb Z/m * \mathbb Z/2$, where $\mathbb Z/m$ is generated by the class $\bar u$ of $u$ and $\mathbb Z/2$ is generated by the class $\bar \Delta$ of $\Delta$.
Now, the order of the image of $\bar u$ under $\overline{\varphi}$ is an odd divisor $l$ of $m$. In particular, $l\neq 1$, because if $\overline{\varphi}(\overline{u})=\id$ then $P=\langle \overline{\varphi}(\overline{\Delta})\rangle$ would be cyclic.
Since $P$ is not abelian, it is not cyclic and we have $l\mid |P|$. Furthermore, if $|P|=l$ then $P=\langle\overline{\varphi}(\overline{u})\rangle$ is abelian, which is absurd. Thus, $|P|\neq l$, hence $|P|\geq 2l\geq 2p=|\W[I_2(p)]|$. Thus, by Lemma \ref{lem-only-group-order2p}, the group $\W[I_2(p)]$ is the smallest non-abelian quotient of $\A[\Gamma]$.

\underline{$\Gamma=I_2(m)$, $m \equiv 2 \pmod{4}$:}
Write $m=2k$, where $k$ is odd.
Let $p$ be the smallest prime divisor of $k$.
Let $s_1, s_2$ be the standard generators of $\A [\Gamma]$.
Set $u=s_1 s_2$, and consider the presentation $\A [\Gamma] = \langle s_1, u \mid u^k s_1 = s_1 u^k \rangle$.
Note that $Z (\A [\Gamma]) = \langle u^k \rangle$ and $\A [\Gamma] / Z (\A [\Gamma]) = \langle \bar s_1, \bar u \mid \bar u^k=\id \rangle \cong \mathbb Z/k * \mathbb Z$.
We consider the presentation $\W [I_2 (p)] = \langle \sigma, \rho \mid \sigma^2 = \rho^p = \id,\ \sigma \rho \sigma = \rho^{-1} \rangle$.
Since $p$ divides $k$, it follows from these presentations that there exists a surjective homomorphism $\varphi : \A [\Gamma] \to \W [I_2 (p)]$ that maps $u$ to $\rho$ and $s_1$ to $\sigma$.

Let $P$ be a finite non-abelian group, and let $\varphi : \A [\Gamma] \to P$ be a surjective homomorphism.
Since $\varphi$ is surjective, it maps the centre onto the centre, hence it induces a homomorphism $\overline{\varphi} : \A [\Gamma] / Z (\A[\Gamma]) \to P/ Z(P)$.
If $\overline{\varphi} (\bar u) = \id$, then $\varphi (u) \in Z (P)$, hence $\varphi (u)$ would commute with $\varphi (s_1)$.
It follows that $P = \langle \varphi(u), \varphi(s_1) \rangle$ would be abelian, a contradiction.
Thus, $\overline{\varphi} (\bar u) \neq \id$.
Since $\bar u$ has order $k$, the element $\overline{\varphi}(\bar u)$ has order $q \ge p$ dividing $k$.
In particular, $q$ divides $|P/Z(P)|$, hence $q$ divides $|P|$.
If $q = |P|$, then $Z (P) = \{\id\}$ and $P = P/Z(P) = \langle \overline{\varphi} (\bar u) \rangle$, which is impossible since $P$ is non-abelian.
So, $|P|>q$, which means $|P| \ge 2q \ge 2p$.
Thus, by Lemma \ref{lem-only-group-order2p}, the group $\W[I_2(p)]$ is the smallest non-abelian quotient of $\A[\Gamma]$.
\end{proof}

\section{Profinite rigidity of irreducible Artin groups of spherical type}\label{subs-intermediate-results}

The aim of this section is to prove Theorem \ref{mainthm}. Instead of writing a lengthy and technical proof, we prove intermediate results and explain how they imply Theorem \ref{mainthm}. Note that our proofs heavily rely on Theorem \ref{thm-SNAQ}, whose proof will be completed in Section \ref{sec-SNAQ-type-E-H}.

The next series of results consists of, as far as we know, new results for which we provide detailed proofs.

\begin{lem}\label{lem-odd-dihedral-not-quotient-B_3}
    Let $m\in\mathbb N_{\geq 5}$ be odd. The group $\W[I_2(m)]$ is not a quotient of $\A[B_3]$. 
\end{lem}

\begin{proof}
Assume that $\W[I_2(m)]$ is a quotient of $\A[B_3]$ and let $\varphi:\A[B_3]\to\W[I_2(m)]$ be an epimorphism. Denote by $\sigma_1,\sigma_2,\sigma_3$ the standard generators of $\A[B_3]$, as in Figure \ref{fig:Cox-graphs-spherical-type}. The images $\varphi(\sigma_1)$ and $\varphi(\sigma_2)$ must be both non trivial. Indeed, if $\varphi(\sigma_i)$ were trivial for $i=1,2$, then $\varphi$ would induce a homomorphism $\overline{\varphi}:\A[B_3]/\llangle \sigma_i\rrangle\longrightarrow \W[I_2(p)]$, which is absurd since $\A[B_3]/\llangle \sigma_i\rrangle\cong \Z$ is abelian for $i=1,2$, and $\W[I_2(p)]$ is not. Moreover, by Remark \ref{rem-generators-odd-conjugate}, $\varphi(\sigma_1)$ and $\varphi(\sigma_2)$ are conjugate, and thus they are both reflections or both rotations.

If they are rotations, they commute, and since they must also satisfy a braid relation, we get $\varphi(\sigma_1)=\varphi(\sigma_2)$. Since $\sigma_3$ commutes with $\sigma_1$, the image $\varphi(\sigma_3)$ must commute with $\varphi(\sigma_1)=\varphi(\sigma_2)$. Hence, the image of $\varphi$ would be abelian, which is a contradiction. Therefore, the elements $\varphi(\sigma_1)$ and $\varphi(\sigma_2)$ must be reflections. Again, if $\varphi(\sigma_1)=\varphi(\sigma_2)$, the image of $\varphi$ would be abelian, so we must have $\varphi(\sigma_1)\neq\varphi(\sigma_2)$. The element $\varphi(\sigma_3)$ commutes with $\varphi(\sigma_1)$, and in odd dihedral groups, we know that the centralizer of a reflection $r$ is $\{1,r\}$. In particular, we obtain $\varphi(\sigma_3)\in\langle\varphi(\sigma_1)\rangle$, and $\W[I_2(m)]=\langle \varphi(\sigma_1),\varphi(\sigma_2)\rangle$. Since we have $\varphi(\sigma_1)\varphi(\sigma_2)\varphi(\sigma_1)=\varphi(\sigma_2)\varphi(\sigma_1)\varphi(\sigma_2)$, the group $\langle \varphi(\sigma_1),\varphi(\sigma_2)\rangle$ is isomorphic to $\W[I_2(3)]$, contradicting $m\geq 5$. This concludes the proof. 
\end{proof}

\begin{lem}\label{lem-odd-dihedral-not-quotient-E-H}
    If $\Gamma$ is of type $E$ or $H$, the group $\A[\Gamma]$ admits no dihedral group as a quotient. 
\end{lem}

\begin{proof}
    Assume that $m\in\mathbb N_{\geq 2}$ is such that $\W[I_2(m)]$ is a quotient of $\A[\Gamma]$, and fix $\varphi:\A[\Gamma]\to\W[I_2(p)]$ an epimorphism.  Since all Artin generators of $\A[\Gamma]$ are conjugate (Remark \ref{rem-generators-odd-conjugate}), the images by $\varphi$ of the Artin generators are either all reflections or all rotations. If they were all rotations, the image of $\varphi$ would be abelian, contradiction. Thus, they are all reflections and the order of the image of an Artin generator by $\varphi$ is $2$, so that $\varphi$ factors through $\W[\Gamma]$. In all cases, the group $\W[\Gamma]$ does not have any dihedral group as a quotient. Indeed, by imposing any dihedral relation between two distinct generators in $\W[\Gamma]$, we find a contradiction. This concludes the proof. 
\end{proof}

\begin{lem}\label{lem-type-B-not-quotient-direct-product}
   Let $n\in\mathbb N_{\geq 3}\backslash\{4\}$ and $G$ be a non-abelian group. The group $G\times \W[A_{n-1}]$ is not a quotient of $\A[B_n]$. 
\end{lem}

\begin{proof}
    Assume for a contradiction that there exists an epimorphism $\varphi:\A[B_n]\to G\times \W[A_{n-1}]$ and write $\overline{\varphi}$ for the restriction of $\varphi$ to $\langle \sigma_1,\dots,\sigma_{n-1}\rangle\cong\A[A_{n-1}]=\B_n$. Moreover, write $\pi_1:G\times\W[A_{n-1}]\to G$ and $\pi_2:G\times\W[A_{n-1}]\to\W[A_{n-1}]$ for the canonical projections. The morphism $\pi_2\circ\varphi:\A[B_n]\to \W[A_{n-1}]$ is surjective and thus, mimicking the proof of Proposition \ref{prop-SNAQ-firstcases} for $\A[B_n]$, we obtain that the morphism $\pi_2\circ\overline{\varphi}$ is surjective (note that we crucially use $n\neq 4$ here, in the form that $\W[A_{n-1}]$ is the smallest non-abelian quotient of $\A[A_{n-1}]$ which is false for $n=4$). But then, since $\A[A_{n-1}]^{\mathrm{ab}}\cong\Z$, this forces $\pi_1\circ \overline{\varphi}$ to be the trivial morphism. This shows that the image of $\pi_1\circ\varphi$ is a quotient of $\A[B_n]/\llangle \sigma_1\rrangle\cong\Z$, contradicting its surjectivity onto $G$ since $G$ is not abelian. This concludes the proof.
\end{proof}

\begin{lem}\label{lem-quotients_of_dihedral} Let $m,p\in\mathbb N_{\geq 3}$ with $p$ odd.
\begin{itemize}
\item[(i)]
If $m$ is divisible by $4$, then the group $\W[I_2(p)]$ is a quotient of $\A[I_2(m)]$.
\item[(ii)]
If $m$ is not divisible by $4$, then the group $\W[I_2(p)]$ is a quotient of $\A[I_2(m)]$ if and only if $p$ divides $m$.
\end{itemize}
\end{lem}

\begin{proof}
Recall that $\A[I_2(m)]$ admits the presentation $\A[I_2(m)]=\langle u,s_1\mid s_1u^k=u^ks_1\rangle$ for $m=2k$ even, and $\A[I_2(m)]=\langle u,\Delta\mid u^m=\Delta^2\rangle$ for $m=2k+1$ odd, with the same notation as in Proposition \ref{prop-SNAQ-firstcases}. Also remind that the centre $Z(\A[I_2(m)])$ is cyclic, generated by $u^k$ if $m=2k$ is even, and by $u^m=\Delta^2$ if $m$ is odd. Thus, we have
    \begin{equation*}
         \A[I_2(m)]/Z(\A[I_2(m)])\cong  \begin{cases} \Z/2*\Z/m \text{ if $m$ is odd},\\
         \Z*\Z/(m/2) \text{ if $m$ is even}.
            
    \end{cases}
    \end{equation*}
 Thus, if $m=4k$, the group $\Z*\Z/2$ is a quotient of $\A[I_2(m)]$, which shows (i) since $\W[I_2(p)]$ is a quotient of $\Z*\Z/2$. This shows (i).

    For (ii), let $m$ be an integer not divisible by $4$ and $p$ be odd. If $p\mid m$ we have the epimorphisms composition $\A[I_2(m)]\to\A[I_2(p)]\to\W[I_2(p)]$, showing the if part of the statement. Now, assume that $\W[I_2(p)]$ is a quotient of $\A[I_2(m)]$ and let $\varphi:\A[I_2(m)]\to \W[I_2(p)]$ be an epimorphism. Since $\W[I_2(p)]$ has trivial centre, the morphism $\varphi$ factors through $\overline\varphi:\A[I_2(m)]/Z(\A[I_2(m)]\to \W[I_2(p)]$.
 Since $m$ is not divisible by $4$, when $m$ is even the integer $m/2$ is odd. Therefore, whether $m$ is even or odd, the group $\A[I_2(m)]/Z(\A[I_2(m)])$ can be generated by two elements $x$ and $y$ where $y$ has odd order dividing $m$. In particular, the element $\overline\varphi(y)\in\W[I_2(p)]$ is a rotation. Thus, since $\overline\varphi$ is surjective, the element $\overline\varphi(x)$ is a reflection. But then, we have $2p=|\W[I_2(p)]|=|\mathrm{im}(\overline\varphi)|=2o(\overline{\varphi(y)})$. In particular, the order of $\overline\varphi(y)$ is equal to $p$. Since this order divides the order of $y$ which in turn divides $m$, this concludes the proof. 
\end{proof}

Recall that the \emph{Euler totient} of an integer $p \ge 2$ is the number $\phi (p)$ of integers between $1$ and $p$ that are coprime to $p$.

\begin{lem}\label{lem-number_epi_I_2(m)_I_2(p)}
    Let $m,p\in\mathbb N_{\geq3}$ with $p$ odd and $m$ divisible by $4$. We have
    \begin{equation}
        \begin{aligned}
            |\mathrm{Epi}(\A[I_2(m)],\W[I_2(p)])|=\begin{cases}
                3p\,\phi(p) & \text{if $p$ divides $m$,}\\
                2p\,\phi(p) & \text{if $p$ does not divide $m$.}
            \end{cases}
        \end{aligned}
    \end{equation}
\end{lem}

\begin{proof}
  Write $m=4k$ and let $\varphi:\A[I_2(m)]\to\W[I_2(p)]$ be an epimorphism. We have \\$\A[I_2(m)]/Z(\A[I_2(m)])\cong\Z*\Z/2k$. Since $\W[I_2(p)]$ has trivial centre, the morphism $\varphi$ factors through $\overline\varphi:\A[I_2(m)]/Z(\A[I_2(m)])\to\W[I_2(p)]$. Write $x$ for a generator of $\Z$ and $y$ for a generator of $\Z/2k$ in the above free product. We separate cases depending on whether $\bar x:=\overline\varphi(x),\bar y :=\overline\varphi(y)$ are reflections or rotations.
  
  \underline{$\bar x$ and $\bar y$ are reflections:}
  In this case, the order of $\bar x$ and $\bar y$ is $2$ and the only constraint is that $\bar x\bar y$ has order $p$. We freely choose the reflection equal to $\bar x$, giving us $p$ choices. Then for $\bar y$, there are $\phi(p)$ reflections for which $\bar x\bar y$ is a rotation of angle $\frac{2\pi}p$. We obtain a contribution of $p\phi(p)$ epimorphisms.

  \underline{$\bar x$ is a rotation and $\bar y$ is a reflection:} We have $p$ choices among reflections for $\bar y$, and $\phi(p)$ choices among rotations for $\bar x$, namely those of order $p$. We obtain a contribution of $p\phi(p)$ epimorphisms. 

  \underline{$\bar x$ is a reflection and $\bar y$ is a rotation:} This case is similar to the previous one, except from the fact that this forces the order $p$ of $\bar y$ to divide $2k$, hence $m$. We obtain a contribution of $p\phi(p)$ epimorphisms if $p\mid m$ and a contribution of $0$ epimorphisms if $p\nmid m$. 

  \underline{$\bar x$ and $\bar y$ are rotations:} In this case, the morphism $\overline\varphi$ is not surjective, contradiction. We obtain a contribution of $0$ epimorphisms.

  Summing up all contributions concludes the proof.
\end{proof}

\begin{lem}\label{lem-internal_rigidity_dihedral}
    Let $m_1,m_2\in\mathbb N_{\geq 3}$. If $\widehat{\A[I_2(m_1)]}=\widehat{\A[I_2(m_2)]}$, then $m_1=m_2$. 
\end{lem}

\begin{proof}
    If $m_1$ is odd and $m_2$ is even the groups $\A[I_2(m_1)]^{\mathrm{ab}}$ and $\A[I_2(m_2)]^{\mathrm{ab}}$ are not isomorphic, hence neither are $\widehat{\A[I_2(m_1)]}$ and $\widehat{\A[I_2(m_2)]}$ by Proposition \ref{prop-profinite-same-abelianization}.

    If $m_1$ and $m_2$ are odd, Theorem \ref{prop-rigidity-one-class-marin} applies.

If $m_1$ and $m_2$ are even and not divisible by $4$ and $\widehat{\A[I_2(m_1)]}\cong\widehat{\A[I_2(m_2)]}$, then $\A[I_2(m_1)]$ and $\A[I_2(m_2)]$ share the same finite quotients by Theorem \ref{thm-same-finite-quotients-same-profinite-completion}. Therefore, Lemma \ref{lem-quotients_of_dihedral} shows that $m_1$ and $m_2$ have the same odd divisors. Since the divisor 2 appears in both $m_1$ and $m_2$ with exponent 1, we obtain $m_1=m_2$. 

    If $m_1$ is even and not divisible by $4$ and $m_2$ is divisible by $4$, there exists $p$ odd such that $\W[I_2(p)]$ is a quotient of $\A[I_2(m_2)]$ but not of $\A[I_2(m_1)]$. Therefore, the groups $\widehat{\A[I_2(m_1)]}$ and $\widehat{\A[I_2(m_2)]}$ are not isomorphic by Theorem \ref{thm-same-finite-quotients-same-profinite-completion}.

    Assume that $m_1$ and $m_2$ are divisible by $4$ and that $\widehat{\A[I_2(m_1)]}\cong\widehat{\A[I_2(m_2)]}$. By Proposition \ref{prop-profinite-same-epis}, for each odd $p$ we have $|\mathrm{Epi}(\A[I_2(m_1)],\W[I_2(p)])|=|\mathrm{Epi}(\A[I_2(m_2)],\W[I_2(p)])|$. But then, Lemma \ref{lem-number_epi_I_2(m)_I_2(p)} shows that $m_1$ and $m_2$ share the same odd divisors. Therefore, without loss of generality, we can assume that $m_1\mid m_2$. If $m_1\neq m_2$, the group $\A[I_2(m_1)]$ is then a strict quotient of $\A[I_2(m_2)]$, contradicting $\widehat{\A[I_2(m_1)]}\cong\widehat{\A[I_2(m_2)]}$ by Lemma \ref{lem-strict-quotient-rigidity}. We obtain $m_1=m_2$. This concludes the proof.

\end{proof}

\begin{lem}\label{lem-A_n_rigidity}
    Let $n\in\mathbb N^*$ and $\Omega$ be an irreducible Coxeter graph of spherical type. If $\widehat{\A[\Omega]}\cong\widehat{\A[A_n]}$, then $\Omega\cong A_n$.
\end{lem}

\begin{proof}

\noindent
\underline{Case $n=1$:} If $\Omega\neq A_1$, the group $\A[A_1]\cong \Z$ is a strict quotient of $\A[\Omega]$ and Lemma \ref{lem-strict-quotient-rigidity} applies.

\underline{Case $n=2,3$:} In these cases, by Theorem \ref{thm-Kolay}, we have $\mathrm{SNAQ}(\A[A_2])=\W[A_2]$. Thus, by Corollary \ref{cor-smallest-diff-profinite-diff}, we have $\mathrm{SNAQ}(\A[\Omega])=\W[A_2]$. Therefore, by Theorem \ref{thm-SNAQ}, we obtain $\\\Omega\in\{A_2,A_3,B_2,B_3,B_4,D_4,F_4\}\cup\{I_2(m)\}_{4\mid m}\cup\{I_2(m)\}_{3\mid m}$. Moreover, since $\A[A_2]^{\mathrm{ab}}\cong \Z$, this is also true for $\A[\Omega]$. This further restricts to $\Omega\in \{A_2,A_3,D_4\}\cup\{I_2(3k)\}_{k \text{ odd}}$. But in this case, Theorem \ref{prop-rigidity-one-class-marin} applies.

\underline{Case $n\geq 4$:} In this case, by Theorem \ref{thm-Kolay} we have $\mathrm{SNAQ}(\A[A_n])=\W[A_n]$. Thus, by Corollary \ref{cor-smallest-diff-profinite-diff}, we have $\mathrm{SNAQ}(\A[\Omega])=\W[A_n]$. Therefore, by Theorem \ref{thm-SNAQ}, we obtain $\Omega\in\{A_n,B_{n+1},D_{n+1}\}$. Moreover, since $\A[A_n]^{\mathrm{ab}}\cong \Z$, this is also true for $\A[\Omega]$. This further restricts to $\Omega=D_{n+1}$, in which case Theorem \ref{prop-rigidity-one-class-marin} applies. This concludes the proof.
\end{proof}

\begin{lem}\label{lem-B_n_rigidity}
    Let $n\in\mathbb N_{\geq 2}$ and $\Omega$ be an irreducible Coxeter graph of spherical type. If $\widehat{\A[\Omega]}\cong\widehat{\A[B_n]}$, then $\Omega\cong B_n$.
\end{lem}

\begin{proof}
    \underline{Case $n=2$:} In this case, by Theorem \ref{thm-SNAQ}, we have $\mathrm{SNAQ}(\A[B_2])=\W[A_2]$. Thus, by Corollary \ref{cor-smallest-diff-profinite-diff}, we have $\mathrm{SNAQ}(\A[\Omega])=\W[A_2]$. Therefore, by Theorem \ref{thm-SNAQ}, we obtain $\Omega\in\{A_2,A_3,B_3,B_4,D_4,F_4\}\cup\{I_2(m)\}_{4\mid m}\cup\{I_2(m)\}_{3\mid m}$. Moreover, since $\A[B_2]^{\mathrm{ab}}\cong \Z^2$, this is also true for $\A[\Omega]$. This further restricts to $\Omega\in \{B_3,B_4,F_4\}\cup\{I_2(m)\}_{4\mid m \text{ or } 6\mid m}$. If $\Omega\in\{B_3,B_4\}$, Corollary \ref{cor-rank-profinite} differentiates $\widehat{\A[B_2]}$ from $\widehat{\A[\Omega]}$. If $\Omega$ is of dihedral type, the case is handled by Lemma \ref{lem-dihedral_rigidity}. Moreover, we verify with GAP that $\W[B_2]$ is a quotient of $\A[B_3]$ but not of $\A[F_4]$. Thus, by Theorem \ref{thm-same-finite-quotients-same-profinite-completion}, the groups $\widehat{\A[B_2]}$ and $\widehat{\A[F_4]}$ are not isomorphic. 

    \underline{Case $n=3$:} In this case, by Theorem \ref{thm-SNAQ}, we have $\mathrm{SNAQ}(\A[B_3])=\W[A_2]$. Thus, by Corollary \ref{cor-smallest-diff-profinite-diff}, we have $\mathrm{SNAQ}(\A[\Omega])=\W[A_2]$. Therefore, by Theorem \ref{thm-SNAQ}, we obtain $\Omega\in\{A_2,A_3,B_2,B_4,D_4,F_4\}\cup\{I_2(m)\}_{4\mid m}\cup\{I_2(m)\}_{3\mid m}$. Moreover, since $\A[B_3]^{\mathrm{ab}}\cong \Z^2$, this is also true for $\A[\Omega]$. This further restricts to $\Omega\in \{B_2,B_4,F_4\}\cup\{I_2(m)\}_{4\mid m \text{ or } 6\mid m}$. If $\Omega\in\{B_2,B_4\}\cup\{I_2(m)\}_{4\mid m \text{ or } 6\mid m}$, Corollary \ref{cor-rank-profinite} differentiates $\widehat{\A[B_3]}$ from $\widehat{\A[\Omega]}$. Moreover, we verify with GAP that $\W[A_2]\times \W[A_2]$ is a quotient of $\A[F_4]$ but not of $\A[B_3]$. Thus, by Theorem \ref{thm-same-finite-quotients-same-profinite-completion}, the groups $\widehat{\A[B_3]}$ and $\widehat{\A[F_4]}$ are not isomorphic. 

    \underline{Case $n=4$:} In this case, by Theorem \ref{thm-SNAQ}, we have $\mathrm{SNAQ}(\A[B_4])=\W[A_2]$. Thus, by Corollary \ref{cor-smallest-diff-profinite-diff}, we have $\mathrm{SNAQ}(\A[\Omega])=\W[A_2]$. Therefore, by Theorem \ref{thm-SNAQ}, we obtain $\Omega\in\{A_2,A_3,B_2,B_3,D_4,F_4\}\cup\{I_2(m)\}_{4\mid m}\cup\{I_2(m)\}_{3\mid m}$. Moreover, since $\A[B_4]^{\mathrm{ab}}\cong \Z^2$, this is also true for $\A[\Omega]$. This further restricts to $\Omega\in \{B_2,B_3,F_4\}\cup\{I_2(m)\}_{4\mid m \text{ or } 6\mid m}$. If $\Omega\in\{B_2,B_3\}\cup\{I_2(m)\}_{4\mid m \text{ or } 6\mid m}$, Corollary \ref{cor-rank-profinite} differentiates $\widehat{\A[B_3]}$ from $\widehat{\A[\Omega]}$. Moreover, we verify with GAP that $\W[F_4]$ is a quotient of $\A[F_4]$ but not of $\A[B_4]$. Thus, by Theorem \ref{thm-same-finite-quotients-same-profinite-completion}, the groups $\widehat{\A[B_4]}$ and $\widehat{\A[F_4]}$ are not isomorphic. 

    \underline{Case $n\geq 5$:} In this case, by Theorem \ref{thm-SNAQ}, we have $\mathrm{SNAQ}(\A[B_n])=\W[A_{n-1}]$. Thus, by Corollary \ref{cor-smallest-diff-profinite-diff}, we have $\mathrm{SNAQ}(\A[\Omega])=\W[A_{n-1}]$. Therefore, by Theorem \ref{thm-SNAQ}, we obtain $\Omega\in\{A_{n-1},D_n\}$. Since $\A[B_n]^{\mathrm{ab}}\cong \Z^2$ and $\A[\Omega]^{\mathrm{ab}}\cong \Z$, Theorem \ref{prop-rigidity-one-class-marin} applies to conclude that $\widehat{\A[B_n]}$ is not isomorphic to $\widehat{\A[\Omega]}$. This concludes the proof. 

\end{proof}

\begin{lem}\label{lem-D_n_rigidity}
    Let $n\in\mathbb N_{\geq 4}$ and $\Omega$ be an irreducible Coxeter graph of spherical type. If $\widehat{\A[\Omega]}\cong\widehat{\A[D_n]}$, then $\Omega\cong D_n$.
\end{lem}

\begin{proof}
    \underline{Case $n=4$:} In this case, by Theorem \ref{thm-SNAQ}, we have $\mathrm{SNAQ}(\A[D_4])=\W[A_2]$. Thus, by Corollary \ref{cor-smallest-diff-profinite-diff}, we have $\mathrm{SNAQ}(\A[\Omega])=\W[A_2]$. Therefore, by Theorem \ref{thm-SNAQ}, we obtain $\Omega\in\{A_2,A_3,B_2,B_3,B_4,F_4\}\cup\{I_2(m)\}_{4\mid m}\cup\{I_2(m)\}_{3\mid m}$. Moreover, since $\A[D_4]^{\mathrm{ab}}\cong \Z$, this is also true for $\A[\Omega]$. This further restricts to $\Omega\in \{A_2,A_3\}\cup\{I_2(3k)\}_{k \text{ odd}}$. But in this case, Theorem \ref{prop-rigidity-one-class-marin} applies.

\underline{Case $n\geq 5$:} In this case, by Theorem \ref{thm-SNAQ} we have $\mathrm{SNAQ}(\A[D_n])=\W[A_{n-1}]$. Thus, by Corollary \ref{cor-smallest-diff-profinite-diff}, we have $\mathrm{SNAQ}(\A[\Omega])=\W[A_{n-1}]$. Therefore, by Theorem \ref{thm-SNAQ}, we obtain $\Omega\in\{A_{n-1},B_{n}\}$. These two cases are handled in Lemmas \ref{lem-A_n_rigidity} and \ref{lem-B_n_rigidity}. This concludes the proof.

\end{proof}

\begin{lem}\label{lem-F_4_rigidity}
    Let $\Omega$ be an irreducible Coxeter graph of spherical type. If $\widehat{\A[\Omega]}\cong\widehat{\A[F_4]}$, then $\Omega\cong F_4$.
\end{lem}

\begin{proof}
     In this case, by Theorem \ref{thm-SNAQ}, we have $\mathrm{SNAQ}(\A[F_4])=\W[A_2]$. Thus, by Corollary \ref{cor-smallest-diff-profinite-diff}, we have $\mathrm{SNAQ}(\A[\Omega])=\W[A_2]$. Therefore, by Theorem \ref{thm-SNAQ}, we obtain $\Omega\in\{A_2,A_3,B_2,B_3,B_4,D_4\}\cup\{I_2(m)\}_{4\mid m}\cup\{I_2(m)\}_{3\mid m}$. Moreover, since $\A[F_4]^{\mathrm{ab}}\cong \Z^2$, this is also true for $\A[\Omega]$. This further restricts to $\Omega\in \{B_2,B_3,B_4\}\cup\{I_2(m)\}_{4\mid m \text{ or } 6\mid m}$. If $\Omega\in\{B_2,B_3,B_4\}$, the case is handled by Lemma \ref{lem-B_n_rigidity}.

     If $\Omega=I_2(m)$ with $4\mid m$, then by Lemma \ref{lem-quotients_of_dihedral} the group $\W[I_2(5)]$ is a quotient of $\A[\Omega]$. However, we verify with GAP that $\W[I_2(5)]$ is not a quotient of $\A[F_4]$, showing that $\A[\Omega]$ and $\A[F_4]$ have non-isomorphic profinite completions by Theorem \ref{thm-same-finite-quotients-same-profinite-completion}. 

     If $\Omega=I_2(m)$ with $6\mid m$ and $4\nmid m$, the group $\W[A_2]\times\W[A_2]$ is not a quotient of $\A[\Omega]$. Indeed, since $\W[A_2]\times\W[A_2]$ has trivial centre, such a quotient has to factor through $\A[\Omega]/Z(\A[\Omega])\cong \Z*\Z/(m/2)$. By assumption, $m/2$ is odd. Therefore, the image of a generator of $\Z/(m/2)$ by such a quotient would need to have order $3$ (because $3$ is the only possible odd order for a non-trivial element in $\W[A_2]\times \W[A_2]$). In particular, such a quotient would factor through $\Z *\Z/3$. But we verify with GAP that $\W[A_2]\times \W[A_2]$ is not a quotient of $\Z*\Z/3$. Since $\W[A_2]\times \W[A_2]$ is a quotient of $\A[F_4]$, the groups $\A[\Omega]$ and $\A[F_4]$ have non-isomorphic profinite completions by Theorem \ref{thm-same-finite-quotients-same-profinite-completion}. This concludes the proof. 
\end{proof}

\begin{lem}\label{lem-dihedral_rigidity}
    Let $m\in\mathbb N_{\geq 5}$ and $\Omega$ be an irreducible Coxeter graph of spherical type. If $\widehat{\A[\Omega]}\cong\widehat{\A[I_2(m)]}$, then $\Omega\cong I_2(m)$.
\end{lem}

\begin{proof}
    If $\Omega$ is of type different from $E$ or $H$ and different from $I_2(m)$, Lemmas \ref{lem-internal_rigidity_dihedral}-\ref{lem-F_4_rigidity} show that $\A[I_2(m)]$ and $\A[\Omega]$ have non-isomorphic profinite completions. If $\Omega$ is of type $E$ or $H$, Theorem \ref{thm-SNAQ} show that $\mathrm{SNAQ}(\A[\Omega])$ is not of dihedral type, in contrast with $\mathrm{SNAQ}(\A[I_2(m)])$. Therefore, Corollary \ref{cor-smallest-diff-profinite-diff} concludes the proof. 
\end{proof}

\begin{lem}\label{lem-E_H_rigidity}
  Let $\Gamma$ be of type $E$ or $H$ and $\Omega$ be an irreducible Coxeter graph of spherical type. If $\widehat{\A[\Gamma]}\cong\widehat{\A[\Omega]}$, then $\Gamma\cong \Omega$.
\end{lem}

\begin{proof}
    An inspection of the list provided in Theorem \ref{thm-SNAQ} shows that if $\Gamma$ is of type $E$ or $H$, the group $\mathrm{SNAQ}(\A[\Gamma])$ is not the smallest non-abelian quotient of any other irreducible Artin group of spherical type. Applying Corollary \ref{cor-smallest-diff-profinite-diff}, this concludes the proof. 
\end{proof}

\begin{prop}\label{prop-2-v-1-rigidity}
Let $\Gamma_1,\Gamma_2$ and $\Omega$ be irreducible Coxeter graphs of spherical type. Then $\widehat{\A[\Gamma_1]}\times\widehat{\A[\Gamma_2]}\ncong\widehat{\A[\Omega]}$.
\end{prop}

\begin{proof}

Let $\Gamma_1,\Gamma_2,\Omega$ be irreducible Coxeter graphs of spherical type. By Proposition \ref{prop-profinite-same-abelianization}, if $\A[\Gamma_1]\times \A[\Gamma_2]$ and $\A[\Omega]$ have isomorphic profinite completions, their abelianizations are isomorphic. Now, the rank of the abelianization of $\A[\Gamma_1]\times\A[\Gamma_2]$ is at least two whereas that of the abelianization of $\A[\Omega]$ is at most two. We obtain $\A[\Gamma_1]^{\mathrm{ab}}\cong\A[\Gamma_2]^{\mathrm{ab}}\cong\Z$ and $\A[\Omega]^{\mathrm{ab}}\cong\Z^2$. Thus, for the remainder of this proof, we shall restrict to this case.

 First, we can immediately rule out the case $\Omega=F_4$. Indeed, it can be checked with GAP that the group $\W[F_4]/Z(\W[F_4])$ is centreless, indecomposable as a direct product, and that its abelianization isomorphic to $(\Z/2)^2$. Since $\A[\Gamma_i]^{\mathrm{ab}}\cong \Z$, the group $\W[F_4]/Z(\W[F_4])$ is not a quotient of $\A[\Gamma_i]$. We now show that $\W[F_4]/Z(\W[F_4])$ is not a quotient of the direct product $\A[\Gamma_1]\times \A[\Gamma_2]$. Given a quotient $\varphi:\A[\Gamma_1]\times \A[\Gamma_2]\to \W[F_4]/Z(\W[F_4])$ and writing $H_i:=\varphi(\A[\Gamma_i])$, we would have $\W[F_4]/Z(\W[F_4])=H_1H_2$ with every element of $H_1$ commuting with every element of $H_2$. In particular, $H_1\cap H_2$ is central, hence trivial. Thus, $\W[F_4]/Z(\W[F_4])\cong H_1\times H_2$, contradiction. We conclude that $\W[F_4]/Z(\W[F_4])$ is a quotient of $\A[F_4]$ but not of $\A[\Gamma_1]\times \A[\Gamma_2]$. Therefore, Theorem \ref{thm-same-finite-quotients-same-profinite-completion} applies and the profinite completions of $\A[\Gamma_1]\times \A[\Gamma_2]$ and $\A[F_4]$ are not isomorphic. From now on, we assume that $\Omega\neq F_4$.

We now deal with the case where $\Gamma_1$ and $\Gamma_2$ are not of type $E$ or $H$. In this case, the groups $\A[\Gamma_1]\times\A[\Gamma_2]$ and $\A[\Omega]$ are good in the sense of Serre by Theorem \ref{prop-not-E-F-H-is-good}. In particular, by Corollary \ref{cor-rank-profinite}, we have $|\Gamma_1|+|\Gamma_2|=|\Omega|$.

    \underline{Case $\Omega$ even dihedral:} In this, case, we need $|\Gamma_1|+|\Gamma_2|=2$, which implies $\Gamma_1=\Gamma_2=A_1$. But then $\A[\Gamma_1]\times\A[\Gamma_2]\cong \Z^2$ is a strict quotient of $\A[\Omega]$ (being its abelianized), hence the profinite completions are not isomorphic by Lemma \ref{lem-strict-quotient-rigidity}.

    \underline{Case $\Omega=B_3$:} In this case, we need $|\Gamma_1|+|\Gamma_2|=3$. Without loss of generality, we can assume that $\Gamma_1=I_2(m)$ for some odd $m$ and that $\Gamma_2=A_1$. Now, by Theorem \ref{thm-SNAQ}, we have $\mathrm{SNAQ}(\A[\Omega])=\W[A_2]$. If $m>3$, the group $\W[I_2(m)]$ is a quotient of $\A[\Gamma_1]\times\A[\Gamma_2]$ but not of $\A[\Omega]$ by Lemma \ref{lem-odd-dihedral-not-quotient-B_3}, hence Theorem \ref{thm-same-finite-quotients-same-profinite-completion} applies. If $m=3$, the group $\A[\Gamma_1]\times\A[\Gamma_2]$ is a strict quotient of $\A[\Omega]$ (obtained by imposing the relation $[\sigma_{2},\sigma_3]=\id$ in $\A[B_n]$). Now Lemma \ref{lem-strict-quotient-rigidity} applies, showing that $\widehat{\A[\Gamma_1]}\times\widehat{\A[\Gamma_2]}\ncong\widehat{\A[\Omega]}$.

    \underline{Case $\Omega=B_4$:} In this case, it can be checked with GAP that the group $\W[B_4]/Z(\W[B_4])$ is centreless, indecomposable and that its abelianization is isomorphic to $(\Z/2)^2$. Thus, arguing in the exact same way as for $\Omega=F_4$, the profinite completions $\A[\Gamma_1]\times\A[\Gamma_2]$ and $\A[\Omega]$ are not isomorphic, because  $\W[B_4]/Z(\W[B_4])$ is a finite quotient of $\A[\Omega]$ but not of $\A[\Gamma_1]\times \A[\Gamma_2]$. 

    \underline{Case $\Omega=B_n$, $n\geq5$:} In this case, by Theorem \ref{thm-SNAQ}, we have $\mathrm{SNAQ}(\A[\Omega])=\W[A_{n-1}]$. Thus, by Corollary \ref{cor-smallest-diff-profinite-diff}, one can assume that at least one of $\A[\Gamma_1]$ or $\A[\Gamma_2]$ has $\W[A_{n-1}]$ as smallest non-abelian quotient. Without loss of generality, Theorem \ref{thm-SNAQ} implies that we need $\Gamma_1\in\{A_{n-1},D_n\}$. Now, since $|\Gamma_1|+|\Gamma_2|=|\Omega|$, the only possibility is $\Gamma_1=A_{n-1}$ and $\Gamma_2=A_1$. Observe as before that $\A[\Gamma_1 ]\times \A[\Gamma_2]$ is a strict quotient of $\A[\Omega]$.
    Indeed, by imposing the relation $[\sigma_{n-1},\sigma_n]=\id$ in $\A[B_n]$, we obtain the direct product of a copy of $\A[A_{n-1}]$ generated by $\sigma_1,\ldots,\sigma_{n-1}$, and a copy of $\A[A_1]$, generated by $\sigma_{n}$.
    Hence the profinite completions are different by Lemma \ref{lem-strict-quotient-rigidity}.

    From now on, assume that $\Gamma_1$ is of type $E$ or $H$. If $\Gamma_2\in\{A_1,E_6,E_7,E_8,H_3,H_4\}$, Theorem \ref{thm-SNAQ} tells us that the smallest non-abelian quotient of $\A[\Gamma_1]\times \A[\Gamma_2]$ is different from $\mathrm{SNAQ}(\A[\Omega])$. Thus, by Corollary \ref{cor-smallest-diff-profinite-diff}, the groups $\widehat{\A[\Gamma_1]}\times\widehat{\A[\Gamma_2]}$ and $\widehat{\A[\Omega]}$ are not isomorphic. Consequently, we can further assume that $\Gamma_2$ is not equal to $A_1$ nor of type $E$ or $H$. In particular, the group $\A[\Gamma_2]$ is good in the sense of Serre. Thus, if the groups $\widehat{\A[\Gamma_1]}\times\widehat{\A[\Gamma_2]}$ and $\widehat{\A[\Omega]}$ are isomorphic, the following inequality holds:
    
    \begin{equation}\label{ineq-cd}
    \begin{aligned}
        |\Omega|&=\mathrm{cd}(\widehat{\A[\Omega]})\\
        &=\mathrm{cd}(\widehat{\A[\Gamma_1]}\times\widehat{\A[\Gamma_2]})\\
        &\geq \mathrm{cd(\widehat{\A[\Gamma_2]}})\\
        &=|\Gamma_2|,
    \end{aligned}
    \end{equation}
where the first and last equalities come from Corollary \ref{cor-rank-profinite} and the inequality is a standard group theoretic fact. We can now complete our case by case analysis.

\underline{Case $\Omega$ even dihedral:} Writing $\Omega=I_2(m)$, we have $\A[\Omega]/Z(\A[\Omega])\cong \Z*\Z/(m/2)$. Given a prime $p$, recall that the subgroup of $\mathrm{GL}_3(\Z/p)$ of unitriangular upper matrices is called the Heisenberg group modulo $p$. This group has order $p^3$ and admits the group presentation $\langle a,b \,|\, a^p=b^p=\id, [a,[a,b]]=[b,[a,b]]=\id\rangle$. Moreover, its abelianization is isomorphic to $(\Z/p)^2$. Given $p$ a prime divisor of $m/2$, the Heisenberg group modulo $p$ is therefore a quotient of $\A[\Omega]$ whose abelianization is $(\Z/p)^2$. Because of the abelianization, this group is not a quotient of $\A[\Gamma_1]$ or $\A[\Gamma_2]$. Therefore, it is not a quotient of $\A[\Gamma_1]\times \A[\Gamma_2]$. Indeed, if it was the case, the images $H_1$ and $H_2$ of $\A[\Gamma_i]$ for $i=1,2$ by such an epimorphism would be subgroups of order $p$ or $p^2$ with the property that every element of $H_1$ commutes with every element of $H_2$. But since groups of order $p$ or $p^2$ are abelian, this would imply that the Heisenberg group modulo $p$ is abelian, which is false. Thus, this group is a quotient of $\A[\Omega]$ but not of $\A[\Gamma_1]\times \A[\Gamma_2]$, hence Theorem \ref{thm-same-finite-quotients-same-profinite-completion} shows that the profinite completions of $\A[\Gamma_1]\times \A[\Gamma_2]$ and $\A[\Omega]$ are not isomorphic.

\underline{Case $\Omega=B_3$:} In this case, Equation \eqref{ineq-cd} tells us that $|\Gamma_2|\leq 3$. Since $\Gamma_2\neq A_1$ by assumption, either $\Gamma_2=I_2(k)$ with $k$ odd or $\Gamma_2=A_3$. First, if $\Gamma_2=I_2(k)$ with $k>3$, the group $\W[I_2(k)]$ is a quotient of $\A[\Gamma_1]\times\A[\Gamma_2]$ but not of $\A[\Omega]$ by Lemma \ref{lem-odd-dihedral-not-quotient-B_3} so that Theorem \ref{thm-same-finite-quotients-same-profinite-completion} applies. \\
If $\Gamma_2\in\{A_2,A_3\}$, the group $\W[A_2]$ is a quotient of $\A[\Gamma_2]$ so that $\W[\Gamma_1]\times \W[A_2]$ is a quotient of $\A[\Gamma_1]\times\A[\Gamma_2]$. But by Lemma \ref{lem-type-B-not-quotient-direct-product}, the group $\W[\Gamma_1]\times \W[A_2]$ is not a quotient of $\A[B_3]$. Therefore, Theorem \ref{thm-same-finite-quotients-same-profinite-completion} applies and the profinite completions are not isomorphic.  

\underline{Case $\Omega=B_4$:} The exact same argument as in the case where $\Gamma_1$ is not of type $E$ or $H$ applies.

\underline{Case $\Omega=B_n$, $n\geq 5$:} In this case, we have $\mathrm{SNAQ}(\A[\Omega])=\W[A_{n-1}]$. Thus, by Corollary \ref{cor-smallest-diff-profinite-diff}, one can assume that $\mathrm{SNAQ}(\A[\Gamma_2])=\W[A_{n-1}]$ and Theorem \ref{thm-SNAQ} implies that $\Gamma_2\in\{A_{n-1},D_n\}$. But then the group $\W[A_{n-1}]$ is a quotient of $\A[\Gamma_2]$ so that the group $\W[\Gamma_1]\times\W[A_{n-1}]$ is a quotient of $\A[\Gamma_1]\times\A[\Gamma_2]$. By Lemma \ref{lem-type-B-not-quotient-direct-product}, the group $\W[\Gamma_1]\times\W[A_{n-1}]$ is not a quotient of $\A[B_n]$. Therefore, Theorem \ref{thm-same-finite-quotients-same-profinite-completion} applies and the profinite completions are not isomorphic. This concludes the proof.
\end{proof}

\begin{proof}[Proof of Theorem \ref{mainthm}]
   Let $\Gamma$ and $\Omega$ be Coxeter graphs such that $\Gamma$ is irreducible and $\widehat{\A [\Gamma]} \cong \widehat{\A [\Omega]}$. 
First, Proposition \ref{prop-profinite-same-abelianization} implies that $\A[\Gamma]^{\mathrm{ab}}\cong\A[\Omega]^{\mathrm{ab}}$. Since $\Gamma$ is irreducible, we have $\A[\Gamma]^{\mathrm{ab}}\in\{\Z,\Z^2\}$. Thus, the graph $\Omega$ has at most two irreducible components.
If $\Omega$ is irreducible, then Lemmas \ref{lem-internal_rigidity_dihedral}--\ref{lem-E_H_rigidity} show that $\Omega \cong\Gamma$. If $\Omega$ has two irreducible components, then the statement is that of Proposition \ref{prop-2-v-1-rigidity}. This concludes the proof. 
\end{proof}

\section{\texorpdfstring{Smallest non-abelian quotients: the cases $E$ and $H$.}{Smallest non-abelian quotients: the cases E and H.}}\label{sec-SNAQ-type-E-H}

\subsection{\texorpdfstring{First reductions for type $E$}{Type E}}\label{sec-reduction-type-E}

In type $E$, we mimic the proof strategy of \cite{Kolay} used for determining $\mathrm{SNAQ}(\A[A_n])$.\\
Let $\Gamma\in\{E_6,E_7,E_8\}$, $S$ be its set of Artin generators and $P$ be a non-abelian quotient of $\A[\Gamma]$. Denote by $p_i$ the image of $\sigma_i$ by this quotient and by $\C$ the conjugacy class of $p_i$ in $P$ (note that since all $p_i$'s are conjugate by Remark \ref{rem-generators-odd-conjugate}, this class does not depend on $i$). By the Orbit-Stabilizer formula, for all $i=1,\dots,|\Gamma|$, we have $|P|=|Z_P(p_i)||\C|$.\\

Our first goal is to find a set $X$ contained in the conjugacy class of the generating set $S$, as large as possible,  satisfying the following property: for every $x,y\in X$, we have $\A[\Gamma]/\llangle xy^{-1}\rrangle\cong \mathbb Z$. The existence of such an $X$ shows that $|\C|\geq |X|$. Indeed, the image of every element of $X$ in $P$ is conjugate to all $p_i$'s, and two such elements are pairwise distinct because $P$ is not abelian.

With this strategy, we are able to prove:

\begin{prop}\label{prop-size-conjugacy-class}
    Let $\Gamma,P$ and $\C$ be as above.
    \begin{enumerate}
        \item[(i)] If $\Gamma=E_6$, we have $|\C|\geq 36$.
    \item[(ii)] If $\Gamma=E_7$, we have $|\C|\geq 63$.
    \item[(iii)] If $\Gamma=E_8$, we have $|\C|\geq 109$. 
    \end{enumerate}

\end{prop}

For proving Proposition \ref{prop-size-conjugacy-class}, we need intermediate lemmas. 

\begin{lem}\label{lem-distinct-artin-generators}
    Let $\Gamma\in \{E_6,E_7,E_8\}$ and $s,t\in S$ with $s\neq t$. Then $\A[\Gamma]/\llangle st^{-1}\rrangle\cong\Z$. 
\end{lem}
\begin{rem}
  Observe that in $\A[A_n]$, if $n\geq 4$, identifying a pair of Artin generators actually identifies all of them. In fact, taking two generators $s\neq t$ in $S$, there exists a third generator $r$ which commutes with one of the two, say $t$, and satisfies a braid relation with the other, say $s$. Identifying $s$ and $t$ then also identifies $r$ with them. Since identifying two adjacent generators (in this case $s$ and $r$) collapses all of them (as mentioned in the proof of Proposition \ref{prop-SNAQ-firstcases}), the remark follows.    
\end{rem}

\begin{proof}[Proof of Lemma \ref{lem-distinct-artin-generators}]
    We know that in $\A[A_n]$, if $n\geq 4$, identifying a pair of Artin generators actually identifies all of them. In $\Gamma$, every pair $\{s,t\}\subset S$ with $s\neq t$ belongs to a copy of $\A[A_n]$ for some $n\in\mathbb N_{\geq 4}$, such that every other generator $u$ belongs to a copy of another $\A[A_m]$ with $m\in \mathbb N_{\geq 4}$ simultaneously with $s$ or $t$ and a third generator $v\in S\cap \A[A_n]$. Thus, identifying $s$ and $t$ also identifies them to $v$, hence with $u$. This concludes the proof.
\end{proof}

Recall that the set of \textit{reflections} of $\A[\Gamma]$ is $\mathcal R[\Gamma]=\{gsg^{-1}\, |\, s\in S, g\in \A[\Gamma]\}$.

\begin{lem}\label{lem-distinct-commuting-reflections}
    Let $\Gamma\in\{E_6,E_7,E_8\}$ and $x,y\in\mathcal R[\Gamma]$ with $x\neq y$. If $xy=yx$, we have $\A[\Gamma]/\llangle xy^{-1}\rrangle\cong\Z$.
\end{lem}

\begin{proof}
    Let $x,y$ be two such reflections. An immediate consequence of the proof \cite[Theorem 2.2]{CGGW19} is that there exists $s,t\in S$ and $g\in \A[\Gamma]$ such that $gxg^{-1}=s$ and $gyg^{-1}=t$. Thus, $\llangle xy^{-1}\rrangle=\llangle st^{-1}\rrangle$ and we conclude using Lemma \ref{lem-distinct-artin-generators}.
\end{proof}

\begin{lem}\label{lem-distinct-reflections-braid}
    Let $\Gamma\in\{E_6,E_7,E_8\}$ and $x,y,z\in\mathcal R[\Gamma]$ with $x\neq z$. If $xz=zx$ and $yzy=zyz$, then $\A[\Gamma]/\llangle xy^{-1}\rrangle\cong \Z$.
\end{lem}

\begin{proof}
    Identifying $x$ with $y$ implies the relations $\bar y\bar z=\bar z\bar y$ and $\bar y\bar z\bar y=\bar z \bar y \bar z$ in the quotient, hence $\bar y=\bar z$. Thus, $x$ and $z$ are identified as well. Since they are different and they commute, we conclude using Lemma \ref{lem-distinct-commuting-reflections}.
\end{proof}

The strategy then consists in attaching reflections in $\A[\Gamma]$ to every root of our Coxeter system, and using Lemmas \ref{lem-distinct-commuting-reflections} and \ref{lem-distinct-reflections-braid} to obtain the largest possible subset $X$ with the property that for all $x,y\in X$ with $x\neq y$, we have $\A[\Gamma]/\llangle xy^{-1}\rrangle\cong\Z$.

We recall briefly what we need about root systems. Let $\Gamma$ be a Coxeter graph and $\M=(m_{s,t})_{s,t\in S}$ be its Coxeter matrix. Let $V=\bigoplus_{s\in S}\mathbb R\,\alpha_s$ a vector space with basis $\Pi=\{\alpha_s\}_{s\in S}$. One can define a symmetric form $\langle .,.\rangle:V\times V\to\R$ by 
\[
\langle \alpha_s, \alpha_t \rangle = \left\{ \begin{array}{ll}
- \cos (\pi/m_{s,t}) & \text{if } m_{s,t} \neq \infty\,,\\
-1 & \text{if } m_{s,t} = \infty\,.
\end{array} \right.
\]

For all $s \in S$, we define the reflection $\rho_s : V \to V$ by 
\[
\rho_s (x) = x - 2 \langle x, \alpha_s \rangle \, \alpha_s\,,\quad x \in V\,.
\]
The representation $\rho : \W [\Gamma] \to \mathrm{GL} (V)$ sending $s$ to $\rho_s$ for all $s \in S$ is then faithful (see \cite{Bourbaki} for more details). From now on, for all $x \in V$ and $w \in W [\Gamma]$, we write $w \cdot x = \rho(w) (x)$.

The \emph{root system} of $\W [\Gamma]$ is the subset $\Phi [\Gamma] = \{ w \cdot \alpha_s \mid w \in \W[\Gamma]\,,\ s \in S \}\subset V$. 
A rood $\beta \in \Phi [\Gamma]$ is \emph{positive} (resp. \emph{negative}) if $\beta = \sum_{s \in S} \lambda_s \alpha_s$ with $\lambda_s \ge 0$ (resp. $\lambda_s \le 0$) for all $s \in S$. 
Write $\Phi^+ [\Gamma]$ for the set of positive roots and $\Phi^- [\Gamma]$ for the set of negative roots. 
We know from \cite{Deodhar} that $\Phi^- [\Gamma] = - \Phi^+ [\Gamma]$ and that $\Phi [\Gamma] = \Phi^+ [\Gamma] \sqcup \Phi^- [\Gamma]$.

\noindent
The set of \textit{refections} of $\W[\Gamma]$ is defined as $R [\Gamma] = \{ wsw^{-1} \mid w \in \W [\Gamma] \,,\ s \in S\}$. 
To any root $\beta \in \Phi [\Gamma]$, one can associate the reflection $r_\beta = w s w^{-1} \in R [\Gamma]$, where $w \in W[\Gamma]$ and $s \in S$ are such that $\beta = w \cdot \alpha_s$. We know from \cite{Deodhar} that $r_\beta$ is independent from such $w$ and $s$. Moreover, given $\beta, \gamma \in \Phi [\Gamma]$, we have $r_\beta = r_\gamma$ if and only if $\gamma = \pm \beta$. In particular, this provides a bijection $\Phi^+ [\Gamma] \to R [\Gamma]$, $\beta \mapsto r_\beta$.

\bigskip\noindent
We have $| R[E_6]| = |\Phi^+ [E_6]| = 36$, $| R[E_7]| = |\Phi^+ [E_7]| = 63$ and $| R[E_8]| = |\Phi^+ [E_8]| = 120$ (see \cite[Planches]{Bourbaki}). Moreover, given $\Gamma \in \{E_6, E_7, E_8\}$, we have $\Phi [\Gamma] \subseteq \bigoplus_{s \in S} \Z \,\alpha_s$. There exists several softwares which allow the computation of roots and positive roots for $E_6,E_7$ and $E_8$. For example, in SageMath (\cite{Sag25}), the following sequence of statements lists the elements of $\Phi^+[E_6]$ as column vectors.

{\small
\begin{verbatim}
R = RootSystem(["E",6])
Phi = R.root_lattice()
PhiPlus1 = Phi.positive_roots()
PhiPlus=[vector(ZZ, alpha).column() for alpha in PhiPlus1]
\end{verbatim}}

\noindent
Let $\beta \in \Phi^+ [\Gamma]$ be a positive root. 
Writing $\beta = w \cdot \alpha_s$ with $s \in S$, $w \in \W [\Gamma]$ and $w$ of minimal length, the quantity $\ell_S(w) +1$ is the \emph{depth} of $\beta$, which we write $\mathrm{dpt} (\beta)$. 
If $w=s_1 s_2 \cdots s_p$ is a reduced expression for $w$, the sequence $(s_1, s_2, \dots, s_p,s)$ of length $\mathrm{dpt} (\beta)$ is called a \emph{reduced expression} of $\beta$. 
Fixing a total order on $S$ (as in $E_6$, $E_7$ and $E_8$), the smallest reduced expression of $\beta$ with respect to the lexicographic order is called the \emph{normal form} of $\beta$, which we write $\mathrm{Norm} (\beta)$.

\noindent
The following Lemma can be stated for any Coxeter graph, but we only need it in cases $E_6,E_7,E_8$, for which the proof is more elementary.

\begin{lem}\label{lem-normal-form-computable}
Let $\Gamma \in \{ E_6, E_7, E_8\}$. There exists an algorithm which, given $\beta \in \Phi^+ [\Gamma]$, computes $\mathrm{Norm} (\beta)$.
\end{lem}

\begin{proof}
    Let $\beta \in \Phi^+ [\Gamma]$. 
First, we have $\mathrm{dpt} (\beta) = 1$ if and only if there exists  $s \in S$ such that $\beta = \alpha_s$.  Furthermore, it is clear that one can determine algorithmically if $\beta = \alpha_s$ for $s \in S$, in which case $s$ can be made explicit. Then $\mathrm{Norm} (\beta) = \mathrm{Norm} (\alpha_s) = (s)$.

\noindent
We now assume $\mathrm{dpt}(\beta)\geq 2$ and proceed by induction on $\mathrm{dpt}(\beta)$. By \cite[Lemma 1.7]{BrinkHowlett}, we know that for all $j$, we have $\mathrm{dpt}(s_j \cdot \beta) = \mathrm{dpt} (\beta)-1$ if and only if $\langle \beta, \alpha_{s_j} \rangle >0$. Let $i$ be the smallest index such that $\langle \beta, \alpha_{s_i} \rangle >0$. Then we have $\mathrm{Norm} (\beta) = (s_i) \cdot \mathrm{Norm}(s_i (\beta))$, where $\cdot$ denotes word concatenation. By induction hypothesis, we can compute $\mathrm{Norm} (s_i (\beta))$, hence $\mathrm{Norm} (\beta)$.
\end{proof}

\noindent
For every $\beta \in \Phi^+ [\Gamma]$, one can define $\kappa (\beta), \kappa' (\beta) \in \mathcal R [\Gamma] \subseteq \A [\Gamma]$ as follows. Let $(s_1, \dots, s_p, s)$ be $\mathrm{Norm}(\beta)$. Then we define
\[
\kappa (\beta) = s_1 s_2 \cdots s_p s s_p^{-1} \cdots s_2^{-1} s_1^{-1}\,,\quad
\kappa' (\beta) = s_1^{-1} s_2^{-1} \cdots s_p^{-1} s s_p \cdots s_2 s_1\,.
\]
Moreover, we define
\[
\mathcal K [\Gamma] = \{ \kappa (\beta) \mid \beta \in \Phi^+ [\Gamma]\}\,,\quad
\mathcal K' [\Gamma] = \{ \kappa' (\beta) \mid \beta \in \Phi^+ [\Gamma]\}\,.
\]
Note that, thanks to Lemma \ref{lem-normal-form-computable}, the sets $\mathcal K [\Gamma]$ and $\mathcal K' [\Gamma]$ can be computed explicitly for $\Gamma \in \{E_6, E_7, E_8\}$. In particular $|\mathcal{K}[\Gamma]|=|\Phi^+[\Gamma]|$, which is equal to 36, 63 or 120 when $\Gamma$ is equal to $E_6, E_7$ or $E_8$, respectively. The following lemma completes the proof of Proposition \ref{prop-size-conjugacy-class}.

\begin{lem}
The following hold.
\begin{enumerate}

\item[(i)] For all $x,y \in \mathcal K [E_6]$ with $x \neq y$, we have $\A [E_6]/ \langle \! \langle xy^{-1} \rangle \! \rangle \cong \Z$.
\item[(ii)] For all $x,y \in \mathcal K [E_7]$ with $x \neq y$, we have $\A [E_7]/ \langle \! \langle xy^{-1} \rangle \! \rangle \cong \Z$. 
\item[(iii)] There exists $X \subseteq \mathcal K [E_8]$ with $|X|=109$, such that for all $x,y \in X$ with $x \neq y$, we have $\A [E_8]/ \langle \! \langle xy^{-1} \rangle \! \rangle \cong \Z$.
\end{enumerate}
\end{lem}

\begin{proof}
Let $\Gamma \in \{E_6, E_7, E_8\}$. This proof relies on SageMath computations. In particular, we programmed the solution to the word problem for $\A[\Gamma]$ based on the word reversing algorithm given in \cite{DehornoyParis}. As we compute it, the set $\Phi^+ [\Gamma]$ consists of vectors, hence it is ordered. This induces an order $\leq$ on $\mathcal K [\Gamma]$ and $\mathcal K' [\Gamma]$. 
We now define multiple subsets of $\mathcal K[\Gamma]\times\mathcal K[\Gamma]$.

\begin{equation*}
\begin{aligned}
     \mathcal A_0 [\Gamma] &= \{ (x,y) \in \mathcal K [\Gamma] \times \mathcal K [\Gamma] \mid \ x < y \},\\
     \mathcal B_1[\Gamma]& = \{ (x,y) \in \mathcal A_0 [\Gamma] \mid x y = yx \},\\
     \mathcal A_1 [\Gamma] &= \mathcal A_0 [\Gamma] \setminus \mathcal B_1 [\Gamma],\\
     \mathcal B_2 [\Gamma] &= \{ (x,y) \in \mathcal A_1 [\Gamma] \mid \exists z \in \mathcal K' [\Gamma] \text{ such that } z \neq x, \ z x = xz,\ y z y = z y z \},\\
     \mathcal A_2 [\Gamma] &= \mathcal A_1 [\Gamma] \setminus \mathcal B_2 [\Gamma],\\
     \mathcal B_3 [\Gamma] &= \{ (x,y) \in \mathcal A_2 [\Gamma] \mid \exists z \in \mathcal K' [\Gamma] \text{ such that } z \neq y, \ z x z = x z x,\ y z = z y \},\\
    \mathcal A_3 [\Gamma] &= \mathcal A_2 [\Gamma] \setminus \mathcal B_3 [\Gamma],\\
    \mathcal B_4 [\Gamma] &= \{ (x,y) \in \mathcal A_3 [\Gamma] \mid \exists z \in \mathcal K [\Gamma] \text{ such that } z \neq x, \ x z = x z,\ y z y = z y z \},\\
    \mathcal A_4 [\Gamma] &= \mathcal A_3 [\Gamma] \setminus \mathcal B_4 [\Gamma],\\
    \mathcal B_5 [\Gamma] &= \{ (x,y) \in \mathcal A_4 [\Gamma] \mid \exists z \in \mathcal K [\Gamma] \text{ such that } z \neq y, \ x z x = z x z,\ y z = z y \},\\
    \mathcal A_5 [\Gamma] &= \mathcal A_4 [\Gamma] \setminus \mathcal B_5 [\Gamma].
    \end{aligned}
\end{equation*}

By Lemma \ref{lem-distinct-commuting-reflections}, if $(x,y)\in \mathcal B_1[\Gamma]$, we have $\A[\Gamma]/\llangle xy^{-1}\rrangle\cong \Z$. Moreover, by Lemma \ref{lem-distinct-reflections-braid}, if $(x,y)\in\mathcal B_2[\Gamma]\cup\mathcal B_3[\Gamma]\cup\mathcal B_4[\Gamma]\cup\mathcal B_5[\Gamma]$, we have $\A[\Gamma]/\llangle xy^{-1}\rrangle\cong\Z$. Using SageMath, we compute $|\mathcal A_5 [E_6]| =0$, $|\mathcal A_5 [E_7]| = 0$  et $|\mathcal A_5 [E_8]| = 11$. Thus, for $\Gamma\in\{E_6,E_7\}$ and for all $x,y\in\mathcal R[\Gamma]$ such that $x\neq y$, we have $\A[\Gamma]/\llangle xy^{-1}\rrangle\cong \Z$. Now, write $\mathcal V = \{ y \in \mathcal K [E_8] \mid \exists x \in \mathcal K [E_8] \text{ such that } (x,y) \in \mathcal A_5 [E_8] \}$ and $X = \mathcal K [E_8] \setminus \mathcal V$. 
Since $|\mathcal A_5 [E_8]| = 11$, we have $|\mathcal V| \le 11$, hence $|X| = |\mathcal K [E_8]| - |\mathcal V| \ge 120 - 11 = 109$. 
But now, for all $x,y \in X$ with $x \neq y$, we have $A [E_8] / \langle \! \langle x y^{-1} \rangle \! \rangle \simeq \Z$. This concludes the proof. 

\end{proof}

Finally, we will extensively use the following result. 

\begin{lem}\cite[Lemma 8]{Kolay}.\label{lem-size-conjugacy-kolay}
    Let $\Omega=A_n$ for $n\geq 2$, $P$ be a non-abelian quotient of $\A[\Omega]$ and $\C$ be the conjugacy class of the image of $\sigma_1$ under this quotient.
    \begin{enumerate}

    \item[(i)] If $n=2,3$, we have $|\C|\geq 3$. 
    \item[(ii)] If $n\geq 4$, we have $|\C|\geq \binom {n+1}2$.
    \end{enumerate}
\end{lem}

We conclude this subsection with a remark on smallest non abelian quotients of Coxeter groups of type $E$.

\begin{rem}\label{rem-SNAQ-W-Case-E} For $\Gamma$ of type $E$, the smallest non abelian quotient of the Coxeter group $\W[\Gamma]$ is $\W[\Gamma]/Z(\W[\Gamma])$ (which is $\W[\Gamma]$ itself for $\Gamma=E_6$). With the same argument as in the proof of Lemma \ref{lem-SNAQ-is-centreless}, we see that SNAQ($\W[\Gamma]$) is centreless, which implies in particular that $Z(\W[\Gamma])$ must be contained in the kernel of the quotient projection. In \cite[Sections 2.8.4 and 3.12.4]{Wilson2009} we can find a decomposition of $\W[\Gamma]$ for $\Gamma=E_7,E_8$ as $
\W[\Gamma]=\W[\Gamma]/Z(\W[\Gamma])\times Z(\W[\Gamma])$, observing that $\W[\Gamma]/Z(\Gamma)$ is simple. Thus SNAQ($\W[\Gamma]$) must be the only non-trivial quotient of $\W[\Gamma]/Z(\W[\Gamma])$, which is the central quotient itself. For $\Gamma=E_6$, the Coxeter group $\W[E_6]$ has trivial centre, but contains a single normal subgroup of index 2 (see again \cite[Sections 2.8.4 and 3.12.4]{Wilson2009}). This assures that SNAQ$(\W[E_6])=\W[E_6]$.
\end{rem}

\subsection{\texorpdfstring{Case $\Gamma=E_6$}{E6}}\label{section-E6}
For this subsection, we fix $\Gamma=E_6$ and $P$ a non-abelian quotient of $\A[\Gamma]$. Write $p_i$ the image of $\sigma_i$ by this quotient. Moreover, denote by $\C$ the conjugacy class of the $p_i$'s.

The aim of this subsection is to show the following result.

\begin{prop}\label{prop-SNAQ-E6}
    The smallest non-abelian quotient of $\A[E_6]$ is $\W[E_6]$. 
\end{prop}

By Proposition \ref{prop-size-conjugacy-class}, we have $|\C|\geq 36$. Moreover, we have $|\W[E_6]|=51\, 840$ and $51\,840/36=1\, 440$. Therefore, to prove Proposition \ref{prop-SNAQ-E6}, it is enough to prove the following.

\begin{prop}\label{prop-order-centralizer-E6}
    We have $|Z_P(p_1)|\geq 1\, 440$, with equality if $P=\W[E_6]$. 
\end{prop}

The proof strategy is as follows. We fix the order $k$ of $p_1$ (equivalently, of all $p_i$'s) and prove a series of lemmas to obtain a lower bound on $|Z_P(p_1)|$ depending on $k$. This will prove the proposition for all cases but $k=3$, which we handle separately. 

Let us first define the groups we need to study and fix some notations.

\begin{defn}

\begin{enumerate}
    \item[(i)] Define $K_1=Z_P(p_1)$ and $k=o(p_1)$.
    \item[(ii)] Define $K_2=\langle p_2,p_4,p_5,p_6\rangle\subset K_1$, $K_3=\frac{K_2}{\langle p_1\rangle\cap K_2}$, $q_i$ to be the image of $p_i$ under the quotient and $k_2$ the order of $q_2$ (equivalently, of all $q_i$'s).
    \item[(iii)] Define $K_4=\langle q_5,q_6\rangle\subset K_3$, $K_5=\frac{K_4}{\langle q_2\rangle \cap K_4}$, $r_i$ to be the image of $q_i$ under the quotient and $k_3$ the order of $r_5$ (equivalently, of both $r_i$'s).
    \end{enumerate}
\end{defn}

\begin{rem}
 For $k=2$, the group $P$ is a quotient of $\W[\Gamma]$. Since $\mathrm{SNAQ}(\W[\Gamma])=\W[\Gamma]$ by Remark \ref{rem-SNAQ-W-Case-E}, we can assume that $k\geq 3$.
\end{rem}

\begin{lem}\label{lem-K3K5-not-abelian-E6}
    The groups $K_3$, $K_4$ and $K_5$ are not abelian.
\end{lem}
\begin{proof} First observe that $\langle p_1\rangle\subseteq Z(K_2)$, thus $\langle p_1\rangle\cap K_2=\langle p_1\rangle\cap Z(K_2)$.
Suppose that $K_3=K_2/(\langle p_1 \rangle \cap Z(K_2))$ is abelian. Then, since its generators $q_2,q_4,q_5,q_6$ commute and satisfy a braid relation, we obtain $q_2=q_4=q_5=q_6$. In particular, the equality $q_4=q_5$ implies that $p_4p_5^{-1}\in \langle p_1 \rangle \cap Z(K_2)$. Hence, $p_4p_5^{-1}$ commutes with $p_2\in K_2$, i.e. $[p_4p_5^{-1},p_2]=p_4p_5^{-1} p_2 p_5 p_4^{-1} p_2^{-1}=\id$. But since $p_2$ commutes with $p_5$, we obtain $[p_2,p_4]=\id$. Now $p_2$ and $p_4$ commute and are conjugated, hence $p_2=p_4$. Therefore, $P$ is a quotient of $\A[E_6]/\llangle \sigma_2\sigma_4^{-1}\rrangle$
, which is abelian by Lemma \ref{lem-distinct-artin-generators}, a contradiction. Then $K_3$ cannot be abelian.

 If $K_4$ was abelian, then we would have $[q_5,q_6]=\id$, which combined with $q_5q_6q_5=q_6q_5q_6$ would imply that $q_5=q_6$ in $K_3$. Thus $p_5p_6^{-1}\in Z(K_2)$, and $p_5p_6^{-1}$ must commute with $p_4\in K_2$, implying $[p_5p_6^{-1},p_4]=\id$ and, consequently, $[p_4,p_5]=\id$. Now $p_4$ and $p_5$ are conjugated and commute, so they must be identified in $K_2$. As before, this would imply that $P$ is abelian, which is absurd.

With a similar argument, assuming $K_5=K_4/(\langle q_2\rangle\cap K_4)=K_4/(\langle q_2\rangle\cap Z(K_4))$ abelian, we find that $r_5$ and $r_6$ coincide, so that $K_5=\langle r_6\rangle\cong \Z$ is cyclic. Then $K_4/Z(K_4)$, which is a quotient of $K_5$, is also cyclic. It is easy to see that if the quotient of a group by its centre is cyclic, then the group is abelian. Since in the previous paragraph we showed that $K_4$ in not abelian, we conclude that $K_5$ is not abelian as well.
 
\end{proof}

\begin{lem}\label{lem-k_1=k_2=k_3-E6}
    We have $k=k_2=k_3$.
\end{lem}

\begin{proof}
    Since $p_i$ is sent to $q_i$ which is sent to $r_i$ under the consecutive quotients, we have $k_3|k_2|k$. Since $r_6^{k_3}=\id$, we have $q_6^{k_3}\in \langle q_2\rangle$, so there exists $n\in\mathbb N$ such that $q_2^n=q_6^{k_3}$. But then $q_6^{k_3}q_2^{-n}=\id$, so $p_6^{k_3}p_2^{-n}\in \langle p_1\rangle$ and there exists $m\in\mathbb N$ such that $p_1^mp_2^n=p_6^{k_3}$. Now, every $\sigma_i$ either commutes with $\sigma_6$ or with $\sigma_1$ and $\sigma_2$. Thus, the element $p_1^mp_2^n=p_6^{k_3}$ is central in $P$. By Lemma \ref{lem-SNAQ-is-centreless}, we get $p_6^{k_3}=\id$. Thus, $k|k_3$ and $k=k_2=k_3$.
\end{proof}

\begin{lem}\label{lem-size-K5-E6}
    We have $|K_5|\geq k^2$.
\end{lem}

\begin{proof}
    We show that $|\{r_5^ir_6^j\}_{i,j=0,\dots,k-1}|=k^2$. Let $0\leq i_1,i_2,j_1,j_2\leq k-1$ such that $r_5^{i_1}r_6^{j_1}=r_5^{i_2}r_6^{j_2}$. Then, we have $r_5^{i_1-i_2}=r_6^{j_2-j_1}$. Therefore, there exists $n,m\in\mathbb N$ such that $p_1^mp_2^np_5^{i_1-i_2}=p_6^{j_2-j_1}$. Now, every $\sigma_i$ either commutes with $\sigma_6$ or $\sigma_1,\sigma_2$ and $\sigma_5$. Thus, the element $p_1^mp_2^np_5^{i_1-i_2}=p_6^{j_2-j_1}$ is central in $P$. By Lemma \ref{lem-SNAQ-is-centreless}, we get $p_6^{j_2-j_1}=\id$. Since $-k<j_2-j_1<k$, this forces $j_1=j_2$ and $r_5^{i_1}=r_5^{i_2}$. By Lemma \ref{lem-k_1=k_2=k_3-E6} the order of $r_5$ is $k$, but $0\leq i_1,i_2\leq k-1$, hence $i_1=i_2$. This concludes the proof.  
\end{proof}

\begin{lem}\label{lem-size-K3-E6}
    We have $|K_3|\geq 10k|K_5|\geq 10k^3$
\end{lem}

\begin{proof}
The group $K_3$ is a quotient of $\A[A_4]=\B_5$, which is the subgroup of $\A[E_6]$ generated by $\sigma_2,\sigma_4,\sigma_5,\sigma_6$. Thus by Lemma \ref{lem-size-conjugacy-kolay} the conjugacy class of $q_2$ is at least $\binom 52=10$. Now, we have $\langle q_2\rangle \cdot K_4\subset Z_{K_3}(q_2)$, hence we obtain 
 \begin{equation*}
 \begin{aligned}
      |Z_{K_3}(q_2)|\geq |\langle q_2\rangle\cdot K_4|&=|\langle q_2\rangle|\left |\frac{\langle K_4\rangle}{\langle q_2\rangle\cap K_4}\right |  \\
      &= k |K_5| \text{ by Lemma \ref{lem-k_1=k_2=k_3-E6}.}
 \end{aligned}
 \end{equation*}
 Therefore, by the Orbit-Stabilizer formula, we obtain $|K_3|\geq 10k|K_5|$. The second inequality follows from Lemma \ref{lem-size-K5-E6}.
\end{proof}

\begin{lem}\label{lem-size-K1-E6}
    We have $|K_1|\geq k|K_3|\geq 10k^2|K_5|\geq 10k^4$. In particular, if $k\geq 4$, we have $|K_1|>1\, 440$.
\end{lem}

\begin{proof}
We have $\langle p_1\rangle\cdot K_2\subset K_1$, hence we obtain 
 \begin{equation*}
 \begin{aligned}
      |K_1|\geq |\langle p_1\rangle\cdot K_2|&=|\langle p_1\rangle|\left |\frac{\langle K_2\rangle}{\langle p_1\rangle\cap K_2}\right |  \\
      &= k |K_3| \text{ by Lemma \ref{lem-k_1=k_2=k_3-E6}.}
 \end{aligned}
 \end{equation*}  
 The second and third inequalities follow from Lemmas \ref{lem-size-K5-E6} and \ref{lem-size-K3-E6}. For $k\geq 4$, we have $10k^4\geq 2\, 560> 1\, 440$, which concludes the proof. 
\end{proof}

The only case which is not handled by Lemma \ref{lem-size-K1-E6} is $k=3$. 

\begin{lem}\label{lem-size-K1-k=3-E6}
    If $k=3$, we have $|K_1|>1\, 440$.
\end{lem}

\begin{proof}
    In this case, the group $\langle p_2,p_4,p_5,p_6\rangle \subset K_1$ is a quotient of $\A[A_4]/\llangle \sigma_1^3\rrangle\cong G_{32}$, where $G_{32}$ denotes one of the exceptional complex reflection groups in the Shephard-Todd notation (see \cite{LehrerTaylor} for more details). The cardinality of $G_{32}$ is $155\,520$, and its smallest non-abelian quotient is the simple group $G_{32}/Z(G_{32})$, of size $29\,920$ (see \cite[Table 1]{BMR}). Thus, we have $|K_1|\geq 29\,920>1\, 440$. This concludes the proof.
\end{proof}

\begin{proof}[Proof of Proposition \ref{prop-SNAQ-E6}]
    Lemmas \ref{lem-size-K1-E6} and \ref{lem-size-K1-k=3-E6} prove Proposition \ref{prop-order-centralizer-E6}, which in turn implies Proposition \ref{prop-SNAQ-E6}.
\end{proof}

\subsection{\texorpdfstring{Case $\Gamma=E_7$}{E7}}

For this subsection, we fix $\Gamma=E_7$ and $P$ a non-abelian quotient of $\A[\Gamma]$. Write $p_i$ the image of $\sigma_i$ by this quotient. Moreover, denote by $\C$ the conjugacy class of the $p_i$'s.

The aim of this subsection is to show the following result.

\begin{prop}\label{prop-SNAQ-E7}
    The smallest non-abelian quotient of $\A[E_7]$ is $\W[E_7]/Z(\W[E_7])$. 
\end{prop}

By Proposition \ref{prop-size-conjugacy-class}, we have $|\C|\geq 63$. Moreover, we have $|\W[E_7]/Z(\W[E_7])|=1\,451\,520$ and  $1\,451\,520/63= 23\,040$. Therefore, to prove Proposition \ref{prop-SNAQ-E7}, it is enough to prove the following.

\begin{prop}\label{prop-order-centralizer-E7}
    We have $|Z_P(p_1)|\geq 23\,040$, with equality if $P=\W[E_7]/Z(W[E_7])$. 
\end{prop}

The proof strategy is as follows. We fix the order $k$ of $p_1$ (equivalently, of all $p_i$') and prove a series of lemmas to obtain a lower bound on $|Z_P(p_1)|$ depending on $k$. This will prove the proposition for all cases but $k=3,4$, which we handle separately. 

Let us first define the groups we need to study and fix some notations.

\begin{defn}

\begin{enumerate}
    \item[(i)] Define $K_1=Z_P(p_1)$ and $k=o(p_1)$.
    \item[(ii)] Define $K_2=\langle p_2,p_4,p_5,p_6,p_7\rangle\subset K_1$, $K_3=\frac{K_2}{\langle p_1\rangle\cap K_2}$, $q_i$ to be the image of $p_i$ under the quotient and $k_2$ the order of $q_2$ (equivalently, all $q_i$'s).
    \item[(iii)] Define $K_4=\langle q_5,q_6,q_7\rangle\subset K_3$, $K_5=\frac{K_4}{\langle q_2\rangle \cap K_4}$, $r_i$ to be the image of $q_i$ under the quotient and $k_3$ the order of $r_5$ (equivalently, all $r_i$'s).
    \end{enumerate}
\end{defn}

\begin{rem}
 For $k=2$, the group $P$ is a quotient of $\W[\Gamma]$. Since $\mathrm{SNAQ}(\W[\Gamma])=\W[\Gamma]/Z(W[\Gamma])$ by Remark \ref{rem-SNAQ-W-Case-E}, we can assume that $k\geq 3$.
\end{rem}

\begin{lem}\label{lem-K3K5-not-abelian-E7}
    The groups $K_3$, $K_4$ and $K_5$ are not abelian.
\end{lem}
\begin{proof}
Suppose that $K_2$ is abelian. Then, $p_2=p_4=p_5=p_6=p_7$, because consecutive generators satisfy at the same time a commutation and a braid relation. This implies that also $p_1=p_2=p_3$, which would mean that $P=\langle p_1\rangle$ is cyclic, a contradiction. Thus $K_2$ is not abelian. Now $K_2$ is a quotient of $\A[A_5]$ (generated by $\sigma_2,\ldots,\sigma_7$), whose abelianization is $\Z$. Thus, Lemma \ref{lem-SNAQ-is-centreless} applies and $K_2/Z(K_2)$ is not abelian. Since $K_2/Z(K_2)$ is a quotient of $K_3$, this shows that $K_3$ is not abelian. 

Similarly, the group $K_4$ is not abelian since $K_3$ is not. Indeed, if $K_4$ was abelian, this would mean that $q_5=q_6=q_7$ and thus that $K_3=\langle q_7\rangle$ is cyclic, which cannot be the case by what we said in the previous paragraph. Now observe that $K_4$ is a quotient of $\langle \sigma_5,\sigma_6,\sigma_7\rangle\cong\A[A_3]=\B_4$, whose abelianization is $\Z$. Thus, Lemma \ref{lem-SNAQ-is-centreless} applies and $K_4/Z(K_4)$ is not abelian. Since $K_4/Z(K_4)$ is a quotient of $K_5$, this shows that $K_5$ is not abelian.
\end{proof}

\begin{lem}\label{lem-k_1=k_2=k_3-E7}
    We have $k=k_2=k_3$.
\end{lem}

\begin{proof}
    Since $p_i$ is sent to $q_i$ which is sent to $r_i$ under the consecutive quotients, we have $k_3|k_2|k$. Since $r_7^{k_3}=\id$, there exists $n\in\mathbb N$ such that $q_2^n=q_7^{k_3}$. But then, there exists $m\in\mathbb N$ such that $p_1^mp_2^n=p_7^{k_3}$. Now, every $\sigma_i$ either commutes with $\sigma_7$ or with $\sigma_1$ and $\sigma_2$. Thus, the element $p_1^mp_2^n=p_7^{k_3}$ is central in $P$. By Lemma \ref{lem-SNAQ-is-centreless}, we get $p_7^{k_3}=\id$. Thus, $k|k_3$ and $k=k_2=k_3$.
\end{proof}

\begin{lem}\label{lem-size-conjugacy-K5-E7}
 The conjugacy class of $r_i$ contains at least $4$ elements, and $r_5\neq r_7$. 
\end{lem}

\begin{proof}
 Observe again that $K_5$ is a quotient of $\langle\sigma_5,\sigma_6,\sigma_7\rangle\cong \A[A_3]=\B_4$. The three elements $r_5,r_6$ and $r_5r_6r_5^{-1}$ are distinct from each other as proved in \cite{Kolay} (or by an immediate hand computation). We claim that $r_7$ is distinct from $r_5,r_6$ and $r_5r_6r_5^{-1}$. First of all, if $r_6=r_7$ the group $K_5$ is abelian, contradicting Lemma \ref{lem-K3K5-not-abelian-E7}. If $r_7=r_5r_6r_5^{-1}$, conjugating by $r_5^{-1}$ yields $r_7=r_6$, again reaching a contradiction.
 
 Assume that $r_7=r_5$. Then, there exists $n\in \mathbb N$ such that $q_2^n=q_5q_7^{-1}$. Thus, there exists $m\in \mathbb N$ such that $p_1^m=p_5p_7^{-1}p_2^{-n}$. Now, every $\sigma_1$ either commutes with $\sigma_1$ or with $\sigma_2,\sigma_5$ and $\sigma_7$. Thus, the element $p_1^m=p_5p_7^{-1}p_2^{-n}$ is central in $P$. By Lemma \ref{lem-SNAQ-is-centreless}, we get $p_5p_7^{-1}p_2^{-n}=\id$. If $n=0$, this implies $p_5=p_7$ and $P$ cyclic, contradiction. Assume $n\geq 1$. If $p_7=p_5p_2^{-n}$, we get
    \begin{equation*}
        \begin{aligned}
            p_6p_7p_6&=p_7p_6p_7\\
            &=p_5p_2^{-n}p_6p_5p_2^{-n}\\
            &=p_2^{-2n}p_5p_6p_5 \text{ since $p_2$ commutes with $p_5$ and $p_6$}\\
            &=p_2^{-2n}p_6p_5p_6.
        \end{aligned}
    \end{equation*}
    Thus, we obtain $p_7=p_2^{-2n}p_5$. Since $p_7=p_2^{-n}p_5$, this implies $p_2^n=\id$ and $p_7=p_5$. Since $\A[E_7]/\llangle \sigma_2^n,\sigma_5\sigma_7^{-1}\rrangle\cong\Z/n$, the group $P$ is a quotient of $\mathbb Z/n$, contradiction. Then the elements $r_5,r_6,r_7$ and $r_5r_6r_5^{-1}$ are all distinct in $K_5$.
\end{proof}

\begin{lem}\label{lem-size-centralizer-K5-E7}
      We have $|Z_{K_5}(r_5)|\geq 2k$. Moreover, if $k$ is prime, we have $|Z_{K_5}(r_5)|\geq k^2$.
\end{lem}

\begin{proof}
      We have $\langle r_5\rangle\cdot \langle r_7\rangle\subset Z_{K_5}(r_5)$. This gives 
  
  \begin{equation*}
      \begin{aligned}
          |Z_{K_5}|\geq |\langle r_5\rangle \cdot \langle r_7\rangle |&=|\langle r_5\rangle|\left |\frac{\langle r_7\rangle }{\langle r_5\rangle\cap \langle r_7\rangle}\right |\ \text{ by Lemma \ref{lem-k_1=k_2=k_3-E7}}\\ &=k\left|\frac{\langle r_7\rangle}{\langle r_5\rangle\cap\langle r_7\rangle}\right|.
      \end{aligned}
  \end{equation*}
 We  claim that $r_7\notin \langle r_5\rangle$, which implies that $\left|\frac{\langle r_7\rangle}{\langle r_5\rangle\cap\langle r_7\rangle}\right|\geq 2$.
 Assume for a contradiction that there exists $1\leq n\leq k-1$ such that $r_7=r_5^n$. This implies that there exists $0\leq m\leq k-1$ such that $q_7=q_5^nq_2^m$. We get
 \begin{equation*}
     \begin{aligned}
         q_6q_5^nq_6q_2^m&=q_6q_5^nq_2^mq_6 \text{ since $q_2$ commutes with $q_6$}\\
         &=q_6q_7q_6\\
         &=q_7q_6q_7\\
         &=q_5^nq_2^mq_6q_5^nq_2^m\\
         &=q_5^nq_6q_5^nq_2^{2m} \text{ Since $q_2$ commutes with $q_5$ and $q_6$.}
     \end{aligned}
 \end{equation*}
Recall that $q_5^nq_6q_5=q_6q_5q_6^n$, thus $q_5^nq_6q_5^n=q_6q_5q_6^nq_5^{n-1}$ and we obtain $q_6q_5^nq_6q_2^m=q_6q_5q_6^nq_5^{n-1}q_2^{2m}$. Equivalently, we have $q_5^{n-1}q_6=q_6^nq_5^{n-1}q_2^m$. Multiplying on the right by $q_5$ yields the equality $q_5^{n-1}q_6q_5=q_6^nq_7$. But $q_5^{n-1}q_6q_5=q_6q_5q_6^{n-1}$. Thus, we obtain $q_5\in\langle q_6,q_7\rangle$. In particular, this implies that $q_5$ commutes with $q_4$ so that $K_3$ is abelian, contradicting Lemma \ref{lem-K3K5-not-abelian-E7}. This shows the first claim. 

If $k$ is prime, then if any non trivial power of $r_7$ belongs to $\langle r_5\rangle$, so does $r_7$. Hence, by the first claim, we have $\langle r_5\rangle \cap \langle r_7\rangle={\id}$ and the second claim follows. This concludes the proof.
\end{proof}

\begin{lem}\label{lem-size-K5-E7}
 We have $|K_5|\geq 8k$. Moreover, if $k$ is prime, we have $|K_5|\geq 4k^2$.
\end{lem}

\begin{proof}
  Combining Lemmas \ref{lem-size-conjugacy-K5-E7} and \ref{lem-size-centralizer-K5-E7}, the proof follows from the Orbit-Stabilizer formula.
\end{proof}

\begin{lem}\label{lem-size-K3-E7}
    We have $|K_3|\geq 15k|K_5|$.
\end{lem}

\begin{proof}
 Since $K_3$ is a quotient of $\A[A_5]=\langle \sigma_2,\sigma_4,\sigma_5,\sigma_6,\sigma_7\rangle$, by Lemma \ref{lem-size-conjugacy-kolay} the conjugacy class of $q_2$ has cardinality at least $\binom 62=15$. Now, we have $\langle q_2\rangle \cdot K_4\subset Z_{K_3}(q_2)$, hence we obtain 
 \begin{equation*}
 \begin{aligned}
      |Z_{K_3}(q_2)|\geq |\langle q_2\rangle\cdot K_4|&=|\langle q_2\rangle|\left |\frac{\langle K_4\rangle}{\langle q_2\rangle\cap K_4}\right |  \\
      &= k |K_5| \text{ by Lemma \ref{lem-k_1=k_2=k_3-E7}.}
 \end{aligned}
 \end{equation*}
 Therefore, by the Orbit-Stabilizer formula, we obtain $|K_3|\geq 15k|K_5|$. 
\end{proof}

\begin{lem}\label{lem-size-K1-E7}
    We have $|K_1|\geq k|K_3|\geq 15k^2|K_5|\geq 120k^3$. Moreover, if $k$ is prime, we have $|K_1|\geq 60k^4$. In particular, if $k\geq 5$, we have $|K_1|>23\, 040$.
\end{lem}

\begin{proof}
We have $\langle p_1\rangle\cdot K_2\subset K_1$, hence we obtain 
 \begin{equation*}
 \begin{aligned}
      |K_1|\geq |\langle p_1\rangle c K_2|&=|\langle p_1\rangle|\left |\frac{\langle K_2\rangle}{\langle p_1\rangle\cap K_2}\right |  \\
      &= k |K_3| \text{ by Lemma \ref{lem-k_1=k_2=k_3-E7}.}
 \end{aligned}
 \end{equation*}  
 The second and third inequalities follow from Lemmas \ref{lem-size-K5-E7} and \ref{lem-size-K3-E7}. For $k\geq 6$, we have $120k^3\geq 25\, 920> 23\, 040$. Since $5$ is prime, for $k=5$, we have $|K_1|\geq 60\times 5^4=37\,500>23\, 040$. This concludes the proof.  
\end{proof}

The only cases which are not handled by Lemma \ref{lem-size-K1-E7} are $k=3,4$. 

\begin{lem}\label{lem-size-K1-k=3-E7}
  If $k=3$, we have $|K_1|>23\, 040$.   
\end{lem}

\begin{proof}
    In this case, the group $\langle p_2,p_4,p_5,p_6\rangle \subset K_1$ is a quotient of $\A[A_4]/\llangle \sigma_1^3\rrangle\cong G_{32}$, where $G_{32}$ denotes one of the exceptional complex reflection groups in the Shephard-Todd notation (see \cite{LehrerTaylor} for more details). The cardinality of $G_{32}$ is $155\,520$, and its smallest non-abelian quotient is the simple group $G_{32}/Z(G_{32})$, of size $29\,920$ (see \cite[Table 1]{BMR}). Thus, we have $|K_1|\geq 29\,920>23\, 040$. This concludes the proof.
\end{proof}

From now on, we assume $k=4$. By Lemma \ref{lem-size-K1-E7}, we have $|K_1|\geq 4|K_3|\geq 240|K_5|$. Therefore, we either need to show that $|K_5|\geq 96=23\, 040/240$, or $|K_3|>5\, 760=20\, 040/4$.

\begin{lem}\label{lem-size-conjugacy-K5-k=4-E7}
 The conjugacy class of $r_7$ is of size at least $6$.
\end{lem}

\begin{proof}
First, by Lemma \ref{lem-size-conjugacy-K5-E7}, the elements $r_5,r_6$ and $r_7$ are different. We show that for $i=1,2,3$, the element $r_7^ir_6r_7^{-i}$ is different from $r_5,r_6,r_7$. This also implies that they are different from each other. \\
If $r_5=r_7^ir_6r_7^{-i}$, conjugating by $r_7^i$ yields $r_5=r_6$. This implies that $K_5$ is abelian, contradicting Lemma \ref{lem-K3K5-not-abelian-E7}. Similarly, $r_7$ is different from these elements. For $i=1,3$, the equality $r_7^ir_6r_7^{-i}=r_6$ is equivalent in $r_7$ and $r_6$ commuting, again reaching a contradiction.

Finally, the equality $r_7^2r_6r_7^{-2}=r_6$ is equivalent to $r_7^2r_6=r_6r_7^2$. We verify with GAP that in $\A[A_3]/\llangle \sigma_1^4, [\sigma_3^2,\sigma_2]\rrangle$, the equality $\sigma_3^2=\sigma_2^2$ holds. This implies that $r_6^2=r_7^2$, hence there exists $n,m\in\mathbb N$ such that $p_7^2=p_6^2p_2^np_1^m$. Now, every $\sigma_i$ either commutes with $\sigma_7$ or with $\sigma_1,\sigma_2$ and $\sigma_6$. Thus, the element $p_7^2=p_6^2p_2^np_1^m$ is central in $P$. By Lemma \ref{lem-SNAQ-is-centreless}, we get $p_7^2=\id$, contradicting the fact that the order of $p_7$ is $4$. This concludes the proof.
\end{proof} 

\begin{lem}\label{lem-size-centralizer-K5-k=4-E7}
If $|Z_{K_5}(r_7)|\leq 16$, then $|K_3|>5\, 760$ or $|K_5|>96$.  
\end{lem}

\begin{proof}
We have $\langle r_5,r_7\rangle\subset Z_{K_5}(r_7)$.

\underline{Claim 1:} If $|\langle r_5,r_7\rangle|<16$, we have $|K_3|>5\, 760$. 

\textit{Proof of Claim 1.} We show that if $|\langle r_5\rangle\cap\langle r_7\rangle|>1$, we have $|K_3|>5\, 760$. The proof of Lemma \ref{lem-size-conjugacy-K5-E7} tells us that if $i$ is equal to $1$ or $3$, the elements $r_5^i$ and $r_7^j$ are different. Moreover, by Lemma \ref{lem-k_1=k_2=k_3-E7}, if $i=0$ or $j=0$ we have $r_5^i\neq r_7^j$. If $i=2$ and $j=1,3$, then $r_7^2=\id$, again reaching a contradiction. This leaves the case $i=j=2$. If $r_5^2=r_7^2$, by definition there exists $0\leq n\leq 3$ such that $q_2^nq_5^2=q_7^2$. If $n=1,3$, then $q_2^2=\id$, contradicting Lemma \ref{lem-k_1=k_2=k_3-E7}. If $n=0$, conjugating $q_5^2=q_7^2$ by $q_5q_4$ shows that $q_4^2=q_7^2$. But then, there exists $m\in\mathbb N$ such that $p_1^mp_4^2=p_7^2$. Now, every $\sigma_i$ either commutes with $p_7$ or with $p_1$ and $p_4$, so that $p_1^mp_4^2=p_7^2$ is central. By Lemma \ref{lem-SNAQ-is-centreless}, we get $p_7^2=\id$, contradicting the fact that the order of $p_7$ is $4$.

The only remaining case is $q_2^2q_5^2=q_7^2$, which we handle by computer using GAP. In this case, the group $K_3$ is a quotient of $\A[A_5]/\llangle \sigma_1^4, \sigma_1^2\sigma_3^2\sigma_5^2\rrangle$. This group is of order $46\, 080$ and its smallest non-abelian quotient where the order of the generators is $4$ is of cardinality $|11\, 520|>|5\, 760|$. This proves Claim 1.

\underline{Claim 2:} If $(r_6r_7)^3\in\langle r_5,r_7\rangle$, we have $|K_5|>96$.

\textit{Proof of Claim 2.} Let $0\leq i,j\leq 3$ be such that $(r_6r_7)^3=r_5^ir_7^j$. Then there exists $n,m\in\mathbb N$ such that $(p_6p_7)^3=p_5^ip_7^jp_2^np_1^m$, or equivalently $(p_6p_7)^3p_1^{-m}=p_5^ip_7^jp_2^n$.
Now, $(p_6p_7)p_1^{-m}$ commutes with every $p_i$ but possibly $p_3$ and $p_5$. Since $p_5^ip_7^jp_2^n$ commutes with $p_3$ and $p_5$, we get that $(p_6p_7)^3p_1^{-m}$ is central in $P$. By Lemma \ref{lem-SNAQ-is-centreless}, we get $(p_6p_7)^3=p_1^m$. But again, this implies that $(p_6p_7)^3=p_1^m$ is central, hence trivial. Thus, we have $(r_6r_7)^3=\id$.
It remains to prove that if $(r_6r_7)^3=\id$ we have $|K_5|>96$, which we handle with GAP. The group $\A[A_3]/\llangle \sigma_1^4,(\sigma_2\sigma_3)^4\rrangle$ is finite of order 192 and admits no quotient in which the images of $\sigma_1$ and $\sigma_3$ are different (necessary because $r_5\neq r_7$) and have order $4$. Thus, in this case, we have $|K_5|>96$. This proves Claim 2.

We proved that if $|K_5|\leq 96$ and $|K_3|\leq 5\, 760$, we have $|Z_{K_5}(r_7)|\geq |\langle r_5,r_7\rangle|+|\{(r_6r_7)^3\}|=17$. This concludes the proof. 
\end{proof}

\begin{lem}\label{lem-size-K1-k=4-E7}
    If $k=4$, we have $|K_1|>23\, 040$. 
\end{lem}

\begin{proof}
If $|K_1|\leq 23\, 040$, then $|K_5|\leq 96$. By Lemmas \ref{lem-size-conjugacy-K5-k=4-E7} and \ref{lem-size-centralizer-K5-k=4-E7}, the Orbit-Stabilizer formula implies that $|K_5|\geq 6\times 17=102>96$, contradiction.
\end{proof}

\begin{proof}[Proof of Proposition \ref{prop-SNAQ-E7}]
    Lemmas \ref{lem-size-K1-E7}, \ref{lem-size-K1-k=3-E7} and \ref{lem-size-K1-k=4-E7} prove Proposition \ref{prop-order-centralizer-E7}, which in turn implies Proposition \ref{prop-SNAQ-E7}.
\end{proof}

\subsection{\texorpdfstring{Case $\Gamma=E_8$}{E8}}

For this subsection, we fix $\Gamma=E_8$ and $P$ a non-abelian quotient of $\A[\Gamma]$. Write $p_i$ the image of $\sigma_i$ by this quotient. Moreover, denote by $\C$ the conjugacy class of the $p_i$'s.

The aim of this subsection is to show the following result.

\begin{prop}\label{prop-SNAQ-E8}
    The smallest non-abelian quotient of $\A[E_8]$ is $\W[E_8]/Z(\W[E_8])$. 
\end{prop}

By Proposition \ref{prop-size-conjugacy-class}, we have $|\C|\geq 109$. Moreover, we have $|\W[E_8]/Z(\W[E_8])|=348\,364\,800$ and  $\lceil348\,364\,800/109\rceil= 3\, 196\, 008$. Therefore, to prove Proposition \ref{prop-SNAQ-E8}, it is enough to prove the following.

\begin{prop}\label{prop-order-centralizer-E8}
    We have $|Z_P(p_8)|\geq 3\, 196\, 008$, with equality if $P=\W[E_8]/Z(W[E_8])$. 
\end{prop}

The proof strategy is as follows. We fix the order $k$ of $p_8$ (equivalently, of all $p_i$') and prove a series of lemmas to obtain a lower bound on $|Z_P(p_8)|$ depending on $k$. This will prove the proposition for all cases but $k=3,4,5,6$, which we handle separately. 

Let us first define the groups we need to study and fix some notations.

\begin{defn}

\begin{enumerate}
    \item[(i)] Define $K_1=Z_P(p_8)$ and $k=o(p_8)$.
    \item[(ii)] Define $K_2=\langle p_1,p_2,p_3,p_4,p_5,p_6\rangle\subset K_1$, $K_3=\frac{K_2}{\langle p_8\rangle\cap K_2}$, $q_i$ to be the image of $p_i$ under the quotient and $k_2$ the order of $q_1$ (equivalently, all $q_i$'s).
    \item[(iii)] Define $K_4=\langle q_1,q_2,q_3,q_4\rangle\subset K_3$, $K_5=\frac{K_4}{\langle q_6\rangle \cap K_4}$, $r_i$ to be the image of $q_i$ under the quotient and $k_3$ the order of $r_1$ (equivalently, all $r_i$'s).
    \item[(iv)] Define $K_6=\langle r_1,r_3\rangle\subset K_5$, $K_7=\frac{K_6}{\langle r_2\rangle \cap K_6}$, $s_i$ to be the image of $r_i$ under the quotient and $k_4$ the order of $s_1$ (equivalently, both $s_i$'s).
    \end{enumerate}
\end{defn}

\begin{rem}
 For $k=2$, the group $P$ is a quotient of $\W[\Gamma]$. Since $\mathrm{SNAQ}(\W[\Gamma])=\W[\Gamma]/Z(\W[\Gamma])$ by Remark \ref{rem-SNAQ-W-Case-E}, we can assume that $k\geq 3$.
\end{rem}

\begin{lem}\label{lem-K3-not-abelian-E8}
    The group $K_3$ is not abelian.
\end{lem}

\begin{proof}
Suppose that $K_2$ is abelian. Then $p_1=p_2=p_3=p_4=p_5=p_6$, because consecutive generators satisfy at the same time a commutation and a braid relation. This implies that also $p_6=p_7=p_8$, which would mean that $P=\langle \sigma_1\rangle$ is cyclic, a contradiction. Thus $K_2$ is not abelian. Now $K_2$ is a quotient of $\A[E_6]$ (generated by $\sigma_1,\dots,\sigma_6$), whose abelianization is $\mathbb Z$. Thus, Lemma \ref{lem-SNAQ-is-centreless} applies and $K_2/Z(K_2)$ is not abelian. Since $K_2/Z(K_2)$ is a quotient of $K_3$, this shows that $K_3$ is not abelian.
\end{proof}

\begin{lem}\label{lem-k_1=k_2-E8}
    We have $k=k_2$.
\end{lem}

\begin{proof}
    Since $p_i$ is sent to $q_i$ we have $k_2|k$. Since $q_1^{k_2}=\id$, there exists $n\in\mathbb N$ such that $p_8^n=p_1^{k_2}$. Now, every $\sigma_i$ either commutes with $\sigma_1$ or $\sigma_8$. Thus, the element $p_8^n=p_1^{k_2}$ is central in $P$. By Lemma \ref{lem-SNAQ-is-centreless}, we get $p_1^{k_2}=\id$. Thus, $k|k_2$ and $k=k_2$.
\end{proof}

Thus, $K_3$ is a non-abelian quotient of $\A[E_6]$ in which the order of the Artin generators is $k$. In particular, all results of Subsection \ref{section-E6} apply. We obtain

\begin{lem}\label{lem-allK-not-abelian-E8}
The groups $K_1,K_3,K_5$ and $K_7$ are not abelian.
\end{lem}

\begin{lem}\label{lem-k_1=k_2=k_3=k_4-E8}
    We have $k=k_2=k_3=k_4$. 
\end{lem}

The following result also holds.

\begin{lem}\label{lem-size-K3-E8}
    We have $|K_3|\geq 36k|K_5|\geq 360k^2|K_7|\geq 360k^4$.
\end{lem}

\begin{proof}
    The group $K_3$ is a non-abelian quotient of $\A[E_6]$. Therefore, by Proposition \ref{prop-size-conjugacy-class}, the conjugacy class of $q_1$ has cardinality at least $36$. Moreover, we showed in Lemma \ref{lem-size-K1-E7} that $|Z_{K_3}(q_1)|\geq k|K_5|\geq 10k^2|K_7|\geq 10k^4$. We conclude using the Orbit-Stabilizer formula. 
\end{proof}

Finally, we have

\begin{lem}\label{lem-size-K1-E8}
    We have $|K_1|\geq k|K_3|\geq 36k^2|K_5|\geq 360k^3|K_7|\geq 360k^5$. In particular, if $k\geq 7$, we have $|K_1|> 3\, 196\, 008$.
\end{lem}

\begin{proof}
    We have $\langle p_8\rangle\cdot K_2\subset K_1$. Thus, we get
    \begin{equation*}
        \begin{aligned}
            |K_1|\geq |\langle p_8\rangle\cdot K_2|&=|\langle p_8\rangle|\left |\frac{K_2}{\langle p_8\rangle\cap K_2}\right|\\
            &=k|K_3|.
        \end{aligned}
    \end{equation*}
We conclude that the claimed inequalities hold by Lemma \ref{lem-size-K3-E8}. For $k\geq 7$, this provides the lower bound $|K_1|\geq 360\times 7^5=6\,050\, 520>3\, 196\, 008$.
\end{proof}

\begin{lem}\label{lem-size-K1-k=3-E8}
    If $k=3$, we have $|K_1|>3\, 196\, 008$.
\end{lem}

\begin{proof}
    By Lemma \ref{lem-size-K1-E8}, we have $|K_1|\geq 36k^2|K_5|=324|K_5|$. Similarly to Lemma \ref{lem-size-K1-k=3-E6}, the group $K_5$ is a non-abelian quotient of $G_{32}$, hence $|K_5|\geq 29\, 920$.\\ We obtain $|K_1|\geq 324\times 29\, 920=9\,694\, 080> 3\, 196\, 008$. This concludes the proof. 
\end{proof}

The remaining cases are more involved, and for clarity we treat them separately. 

\subsubsection{Case k=4}

In this subsection, we fix $k=4$. By Lemma \ref{lem-size-K1-E8}, we have $|K_1|\geq 4|K_3|\geq 576|K_5|$. Therefore, we either need to show that $|K_3|\geq 799\, 002 =3\, 196\, 008/4$, or that $|K_5|\geq 5\, 454=\lceil 3\, 196\, 008/576\rceil$.

The strategy is as follows. The group $K_5$ is a quotient of the truncated braid group $B_5(4)=\A[A_4]/\llangle \sigma_i^4\rrangle$, whose derived subgroup $B_5(4)'$ is perfect. In particular, any non-abelian quotient of $B_5(4)$ has a perfect derived subgroup. We do computations with GAP to identify the potential small such perfect groups, and find contradictions for each of them. 

For the remainder of this subsection, we fix $Q$ a non-abelian quotient of $B_5(4)$. 

\begin{lem}\label{lem-perfect-quotient-derived-B_5(4)'}
    There are $5$ prefect quotients of $B_5(4)'$ of order less than or equal to $5\, 454$.
\end{lem}

\begin{proof}
    We obtain a presentation of $B_5(4)'$ using Reidemeister-Schreier algorithm, then with GAP we do an extensive search of epimorphisms onto perfect groups of order less than or equal to $5\, 454$. 
\end{proof}

The list of these $5$ quotients consists of the alternating group $\mathfrak A_5$, two semi-direct products of $(\Z/2)^4$ by $\mathfrak A_5$ of order 960 and, for each of these semi-direct products, a non split central extension by $\Z/2$ of the group. In GAP, these $5$ groups are respectively called $\pg(60,1),\pg(960,1)$,\\ $\pg(960,6),\pg(1\, 920,1)$ and $\pg(1\, 920, 2)$. 

\begin{lem}\label{lem-pg60-1-not-quotient-E8}
The group $Q'$ is not isomorphic to $\pg(60,1)$.     
\end{lem}

\begin{proof}
Using GAP, one can compute the normal subgroups $N$ of $B_5(4)'$ for which $B_5(4)'/N\cong\mathfrak A_5$. In fact, there is only $1$. If $K_5'$ is isomorphic to $\mathfrak A_5$, then $K_5$ is a quotient of $B_5(4)/\llangle N\rrangle$ and we can verify with GAP that $\llangle N\rrangle$ contains $r_1^2r_3^{-2}$. In particular, we have $r_1^2=r_3^2$. Thus, there exists $n,m\in\mathbb N$ such that $p_1^2=p_3^2p_6^np_8^m$. Now, every $\sigma_i$ either commutes with $\sigma_1$ or with $\sigma_3,\sigma_6$ and $\sigma_8$. This shows that $p_1^2=p_3^2p_6^np_8^m$ is central in $P$. By Lemma \ref{lem-SNAQ-is-centreless}, we get $p_1^2=\id$, contradiction.
\end{proof}

\begin{lem}\label{lem-pg960-1-not-quotient-E8}
    The group $Q'$ is not isomorphic to $\pg(960,1)$.
\end{lem}

\begin{proof}
   Using GAP, one can compute the normal subgroups $N$ of $B_5(4)'$ for which $B_5(4)'/N\cong\pg(960,1)$. In fact, there is only $1$. If $K_5'$ is isomorphic to $\mathfrak \pg(960,1)$, then $K_5$ is a quotient of $B_5(4)/\llangle N\rrangle$ and we can verify with GAP that $\llangle N\rrangle$ contains $r_3r_1r_2^2r_1r_3r_1^2$. In particular, there exists $n\in\mathbb N$ such that $q_3q_1q_2^2q_1q_3q_1^2=q_6^n$. For $n=1,3$, we verify with GAP that this implies that $K_3$ is abelian, contradicting Lemma \ref{lem-allK-not-abelian-E8}. For $n=2$, we verify with GAP that this implies that the order of the $q_i$'s is $2$, contradicting Lemma \ref{lem-k_1=k_2=k_3=k_4-E8}. For $n=0$, we verify with GAP that this implies that $q_1^2=q_3^2$. But then, there exists $m\in\mathbb N$ such that $p_1^2=p_3^2p_8^m$. Now, every $\sigma_i$ either commutes with $\sigma_1$ or with $\sigma_3$ and $\sigma_8$. This shows that $p_1^2=p_3^2p_8^m$ is central in $P$. By Lemma \ref{lem-SNAQ-is-centreless}, we get $p_1^2=\id$, contradiction. \\
\end{proof}

\begin{lem}\label{lem-pg960-6-not-quotient-E8}
    If the group $Q'$ is isomorphic to $\pg(960,6)$, then $|K_3|>799\,002=3\, 196\, 008/4$.
\end{lem}

\begin{proof}
      Using GAP, one can compute the normal subgroups $N$ of $B_5(4)'$ for which $B_5(4)'/N\cong\pg(960,6)$. In fact, there is only $1$. If $K_5'$ is isomorphic to $\mathfrak \pg(960,6)$, then $K_5$ is a quotient of $B_5(4)/\llangle N\rrangle$ and we can verify with GAP that $\llangle N\rrangle$ contains $(r_2r_3r_4^2)^2$. In particular, there exists $n\in\mathbb N$ such that $q_3q_1q_2^2q_1q_3q_2^2=q_6^n$. For $n=1,3$, we verify with GAP that this implies that $K_3$ is abelian, contradicting Lemma \ref{lem-allK-not-abelian-E8}. For $n=2$, we verify with GAP that this implies that the order of the $q_i$'s is $2$, contradicting Lemma \ref{lem-k_1=k_2=k_3=k_4-E8}. For $n=0$, this implies that $K_3$ is a quotient of $\A[E_6]/\llangle \sigma_1^4,(\sigma_2\sigma_3\sigma_4^2)^2\rrangle$. We verify with GAP that this group has order $6\, 635\, 520$ and $3$ non-abelian quotients in which the images of the Artin generators have order $4$, including itself. Among these $3$ quotients, two have order more than $799\, 002$ and one has order $103\, 680$. In the quotient of order $103\, 680$, the equality $q_1^2=q_3^2$ holds. But if $q_1^2=q_3^2$, there exists $m\in\mathbb N$ such that $p_1^2=p_3^2p_8^m$. Similarly to Lemma \ref{lem-pg960-1-not-quotient-E8}, this leads to a contradiction. This shows that $|K_3|>799\, 002$, which concludes the proof. 
\end{proof}

\begin{lem}\label{lem-pg1920-not-quotient-E8}
    If $Q'$ is isomorphic to $\pg(1\, 920,1)$ or $\pg(1\, 920, 2)$, we have $|K_5|>5\, 454$. 
\end{lem}

\begin{proof}
    Using GAP, one can compute the normal subgroups $N$ of $B_5(4)'$ for which $B_5(4)'/N\cong\pg(1920,1)$. In fact, there is only one. The group $B_5(4)/\llangle N\rrangle$ has order $7\, 680$. Again with GAP, we verify that all strict quotients of this group have derived subgroup of order strictly less than $1\, 920$. Thus, by assumption, we have $K_5\cong B_5(4)/\llangle N\rrangle$ and $|K_5|>5\, 454$. The proof is exactly similar for $\pg(1\, 920,2)$. 
\end{proof}

\subsubsection{\texorpdfstring{Case $k=5$}{Case k=5}}

In this subsection, we fix $k=5$. By Lemma \ref{lem-size-K1-E8}, we have $|K_1|\geq 360k^3|K_7|=45\, 000|K_7|$. Therefore, we need to show that $|K_7|\geq 72=\lceil 3\, 196\, 008/45\, 000 \rceil$.

\begin{lem}\label{lem-(s_1s_3)^3-not-1-k=5-E8}
    We have $(s_1s_3)^3\neq \id$. In particular, the centre of $K_7$ is not trivial. 
\end{lem}

\begin{proof}
    If $(s_1s_3)^3=\id$, there exists $n,m,t\in\mathbb N$ such that $(p_1p_3)^3=p_2^np_6^mp_8^t$. Now, every $\sigma_i$ either commutes with $\sigma_1,\sigma_3$ and $\sigma_2$ or with $\sigma_6$ and $\sigma_8$. This shows that $(p_1p_3)^3p_2^{-n}$ is central in $P$. By Lemma \ref{lem-SNAQ-is-centreless}, we obtain $(p_1p_3)^3=p_2^n$. For $0\leq n\leq 4$, we verify with GAP that the group $\A[E_8]/\llangle \sigma_1^5, (\sigma_1\sigma_3)^3\sigma_2^{-n}\rrangle$ is trivial, contradiction. Thus $(s_1s_3)^3\neq \id$, showing in particular that $Z(K_7)\neq \{\id\}$ since $(s_1s_3)^3$ is central in $K_7$. This concludes the proof. 
\end{proof}

\begin{lem}\label{size-K7-k=5-E8}
  If $k=5$, we have $|K_7|>72$. 
\end{lem}

\begin{proof}
    The group $K_7$ is a quotient of $\A[A_2]/\llangle \sigma_1^5\rrangle\cong G_{16}$, where $G_{16}$ denotes one of the exceptional complex reflection groups in the Shephard-Todd notation (\cite{LehrerTaylor}). The group $G_{16}$ has one non-abelian quotient of order less than or equal to $72$, namely $G_{16}/Z(G_{16})\cong\mathfrak A_5$ (see \cite{BMR}). But $K_7$ is not isomorphic to $\mathfrak A_5$ since it has non trivial centre by Lemma \ref{lem-(s_1s_3)^3-not-1-k=5-E8}. This concludes the proof. 
\end{proof}

\subsubsection{\texorpdfstring{Case $k=6$}{Case k=6}}

In this subsection, we fix $k=6$. By Lemma \ref{lem-size-K1-E8}, we have $|K_1|\geq 360k^2|K_7|=77\,760$. Therefore, we need to show that $|K_7|\geq 42=\lceil 3\, 196\, 008/77\, 760\rceil$.

\begin{lem}\label{lem-(s_1s_3)^3-not-1-k=6-E8}
    We have $(s_1s_3)^3\neq \id$. In particular, the centre of $K_7$ is not trivial. 
\end{lem}

\begin{proof}
    If $(s_1s_3)^3=\id$, there exists $n,m,t\in\mathbb N$ such that $(p_1p_3)^3=p_2^np_6^mp_8^t$. Now, every $\sigma_i$ either commutes with $\sigma_1,\sigma_3$ and $\sigma_2$ or with $\sigma_6$ and $\sigma_8$. This shows that $(p_1p_3)^3p_2^{-n}$ is central in $P$. By Lemma \ref{lem-SNAQ-is-centreless}, we obtain $(p_1p_3)^3=p_2^n$. For $1\leq n\leq 5$, we verify with GAP that $\A[E_6]/\llangle \sigma_1^6, (\sigma_1\sigma_3)^3\sigma_2^{-n}\rrangle$ is abelian. Thus, for $1\leq n\leq 5$, the equality $(p_1p_3)^3p_2^{-n}=\id$ does not hold by Lemma \ref{lem-distinct-artin-generators}. For $n=0$, the group $\A[E_6]/\llangle \sigma_1^6,(\sigma_1\sigma_3)^3\rrangle$ is not abelian but we can verify with GAP that $\sigma_1^2$ commutes with $\sigma_2$, which implies that $p_1^2$ commutes with $p_2$. But now, the element $p_1^2$ is central in $P$. By Lemma \ref{lem-SNAQ-is-centreless}, we get $p_1^2=\id$, contradiction.
\end{proof}

\begin{lem}\label{lem-(s_1s_3)^3-not-s_1^is_3^j-k=6-E8}
        We have $(s_1s_3)^3\notin \langle s_1\rangle\cdot\langle s_3\rangle$. 

\end{lem}

\begin{proof}
    If $(s_1s_3)^3=s_1^is_3^j$, then $K_7$ is a quotient of $P(i,j):=\A[A_2]/\llangle \sigma_1^6, (\sigma_1\sigma_3)^3\sigma_3^{-j}\sigma_1^{-i}\rrangle$. We verify with GAP that for $(i,j)\notin\{(0,0),(0,3),(0,4),(2,4),(3,0),(3,3),(4,0),(4,2),\}$, the group $P(i,j)$ is abelian, contradicting Lemma \ref{lem-allK-not-abelian-E8}. The case $(i,j)=(0,0)$ is treated separately in Lemma \ref{lem-(s_1s_3)^3-not-1-k=6-E8}. \\
    \fbox{Case $(i,j)\in\{(0,4),(4,0)\}$:} If $(s_1s_3)^3=s_1^is_3^j$, we verify with GAP that the order of $s_1$ is $2$, contradicting Lemma \ref{lem-k_1=k_2=k_3=k_4-E8}.\\
    \fbox{Case $(i,j)\in\{(0,3),(3,0)\}$:} If $(s_1s_3)^3=s_1^is_3^j$, by definition there exists $n\in\mathbb N$ such that $(r_1r_3)^3=r_1^ir_3^jr_2^n$. For $n=1,3,5$, we verify with GAP that this implies that $K_5$ is abelian, contradicting Lemma \ref{lem-allK-not-abelian-E8}. For $n=0,4$, we verify with GAP that this implies that the order of $r_1$ is $2$, contradicting Lemma \ref{lem-k_1=k_2=k_3=k_4-E8}. For $n=2$, we verify with GAP that this implies that $r_1^2=r_2^2$. Thus, there exists $m,t$ such that $p_1^2=p_2^2p_6^mp_8^t$. Now, every $\sigma_i$ either commutes with $\sigma_1$ or with $\sigma_2,\sigma_6$ and $\sigma_8$. This shows that $p_1^2=p_2^2p_6^mp_8^t$ is is central in $P$. By Lemma \ref{lem-SNAQ-is-centreless}, we get $p_1^2=\id$, contradiction. \\
    \fbox{Case $(i,j)\in\{(2,4),(4,2)\}$:}  If $(s_1s_3)^3=s_1^is_3^j$, by definition there exists $n\in\mathbb N$ such that $(r_1r_3)^3=r_1^ir_3^jr_2^n$. For $n=1,3,5$, we verify with GAP that this implies that $K_5$ is abelian, contradicting Lemma \ref{lem-allK-not-abelian-E8}. For $n=2,4$, we verify with GAP that this implies that the order of $r_1$ is $2$, contradicting Lemma \ref{lem-k_1=k_2=k_3=k_4-E8}. For $n=0$, we verify with GAP that this implies that $r_1^2=r_2^2$. Thus, there exists $m,t$ such that $p_1^2=p_2^2p_6^mp_8^t$. Now, every $\sigma_i$ either commutes with $\sigma_1$ or with $\sigma_2,\sigma_6$ and $\sigma_8$. This shows that $p_1^2=p_2^2p_6^mp_8^t$ is is central in $P$. By Lemma \ref{lem-SNAQ-is-centreless}, we get $p_1^2=\id$, contradiction. \\
    \fbox{Case $(i,j)=(3,3)$:} If $(s_1s_3)^3=s_1^3s_3^3$, by definition there exists $n\in\llbracket 0,5\rrbracket$ such that $(r_1r_3)^3=r_1^3r_3^3r_2^n$. We verify with GAP that in all cases, this implies that $K_5$ is abelian, contradicting Lemma \ref{lem-allK-not-abelian-E8}. This concludes the proof.
    
\end{proof}

\begin{lem}\label{lem-size-K7-k=6-E8}
    If $k=6$, we have $|K_7|\geq 72$. 
\end{lem}

\begin{proof}
     We know from Lemma \ref{lem-size-K5-E6} that $|\langle s_1\rangle\cdot\langle s_3\rangle|=36$. Moreover, if $s_1^{i_1}s_3^{j_1}=s_1^{i_2}s_3^{j_2}(s_1s_3)^3$, then $s_1^{i_1-i_2}s_3^{j_1-j_2}=(s_1s_3)^3$, contradicting Lemma \ref{lem-(s_1s_3)^3-not-s_1^is_3^j-k=6-E8}. Thus, we have $$|K_7|\geq |\langle s_1\rangle\cdot\langle s_3\rangle\sqcup \langle s_1\rangle\cdot\langle s_3\rangle\cdot(s_1s_3)^3|=72.$$ This concludes the proof. 
\end{proof}

\begin{proof}[Proof of Proposition \ref{prop-SNAQ-E8}]
   Lemmas \ref{lem-size-K1-E8}, \ref{lem-size-K1-k=3-E8}, \ref{lem-pg60-1-not-quotient-E8}, \ref{lem-pg960-1-not-quotient-E8}, \ref{lem-pg960-6-not-quotient-E8}, \ref{lem-pg1920-not-quotient-E8} and \ref{lem-size-K7-k=6-E8} prove Proposition \ref{prop-order-centralizer-E8}, which in turn implies Proposition \ref{prop-SNAQ-E8}.
\end{proof}

\subsection{\texorpdfstring{Case $\Gamma=H_3$}{H3}}

For this subsection, we fix $\Gamma=H_3$ and $P$ a non-abelian quotient of $\A[\Gamma]$.

\begin{prop}\label{prop-SNAQ-H3}
    The smallest non-abelian quotient of $\A[\Gamma]$ is $\W[\Gamma]/Z(\W[\Gamma])\cong \mathfrak A_5$. 
\end{prop}

\begin{proof}
    We know by Lemma \ref{lem-SNAQ-is-centreless} that $\mathrm{SNAQ}(\A[\Gamma])$ is centreless. Thus, $P$ factors through $\A[\Gamma]/Z(\A[\Gamma])$. We verify with GAP that the derived subgroup of $\A[\Gamma]/Z(\A[\Gamma])$ is perfect. Thus, the group $P'$ is perfect and non trivial, since $P$ is not abelian. The smallest perfect group is $\mathfrak A_5$, which is a quotient of $\A[\Gamma]/Z(\A[\Gamma])$. This concludes the proof. 
\end{proof}
\subsection{\texorpdfstring{Case $\Gamma=H_4$}{H4}}

For this subsection, we fix $\Gamma=H_4$ and $P$ a non-abelian quotient of $\A[\Gamma]$. Write $p_i$ the image of $\sigma_i$ by this quotient. Moreover, denote by $\C$ the conjugacy class of the $p_i$'s. Finally, denote by $c$ the element $p_1p_2p_3p_4$. 

The aim of this subsection is to show the following result.

\begin{prop}\label{prop-SNAQ-H4}
    The smallest non-abelian quotient of $\A[\Gamma]$ is $\W[\Gamma]/Z(\W[\Gamma])$.
\end{prop}

In order to get the most out of the Orbit-Stabilizer formula, we start by finding a lower bound on $|\C|$. 

\begin{lem}\label{lem-size-conjugacy-H4}
    We have $|\C|\geq 30$.
\end{lem}

\begin{proof}
    We verify with GAP that the set $\{c^ip_jc^{-i}\}_{i=0,\dots,14,j=1,\dots,4}$ is of cardinality $30$. More precisely, given $i\in \llbracket 0,14\rrbracket$,  the elements $c^ip_2c^{-i},c^{i+14}p_3c^{-i-14}$ and $c^{i+13}p_4c^{-i-13}$ are equal.\\
    The verification is as follows. Since $\mathrm{SNAQ}(\A[H_4])$ is centreless by Lemma \ref{lem-SNAQ-is-centreless}, any epimorphism $\A[H_4]\to \mathrm{SNAQ}(\A[H_4])$ factors through $\A[H_4]/Z(\A[H_4])=\A[H_4]/\langle c^{15}\rangle$.  
    For all 
    $0\leq i\leq 14$ and $j_1\leq j_2$, we consider the quotient of $\A[H_4]/Z(\A[H_4])$ by the relation $\sigma_{j_1}=c^i\sigma_{j_2}c^{-i}$. This quotient is either abelian, contradicting the hypothesis, or isomorphic to $\A[H_4]/Z(\A[H_4])$, in which case the relation holds in all centreless quotients of $\A[H_4]$. If $\sigma_{j_1}=c^i\sigma_{j_2}c^{-i}$, then for all $i'\in\llbracket 0,14\rrbracket$ we have $ c^{i'}\sigma_{j_1}c^{-i'}=c^{i+i'}\sigma_{j_2}c^{-i-i'}$. In our verification, we find that $p_2=c^{14}p_3c^{-14}=c^{13}p_4c^{-13}$ and $p_3=p^{14}p_4p^{-14}$. Now, since $c^{15}=\id$, the equality $p_3=c^{14}p_4c^{-14}$ is of the form $c^{14+i'}p_3c^{-14-i'}=c^{13+i'}p_4c^{-13-i'}$. Therefore, every element of the form $c^ip_1c^{-i}$ is distinct from all other elements of $\{c^ip_jc^{-i}\}_{i=0,\dots,14,j=1,\dots,4}$, which accounts for 15 elements, and we have 15 other distinct elements which are $\{c^ip_2c^{-i}\}_{i=0,\dots,14}$. This concludes the proof.

\end{proof}

Since $|\W[H_4]/Z(\W[H_4])|=7\, 200$ and $\lceil 7\, 200/30\rceil=240$, Lemma \ref{lem-size-conjugacy-H4} shows that it is enough to prove the following result.

\begin{prop}\label{prop-order-centralizer-H4}
We have $|Z_P(p_4)|\geq 240$, with equality if and only if $P\cong \W[H_4]/Z(\W[H_4])$.
\end{prop}

The proof strategy is as follows. We fix the order $k$ of $p_1$ (equivalently, of all $p_i$'s) and prove a series of lemmas to obtain a lower bound on $|Z_P(p_4)|$ depending on $k$. This will prove the proposition for all cases but $k=3,4,5$, which we handle separately. The case $k=2$ is the case where $P$ is a quotient of $\W[H_4]$, whose smallest non-abelian quotient is $\W[H_4]/Z(\W[H_4])$. Therefore, we assume $k\geq 3$.  

Let us first define the groups we need to study and fix some notations. 

\begin{defn}
    \begin{itemize}
        \item[(i)] Define $K_1=Z_P(p_4)$ and $k=o(p_4)$. 
        \item[(ii)] Define $K_2=\langle p_1,p_2\rangle$, $K_3=\frac{K_2}{\langle p_4\rangle\cap K_2}$, $q_i$ to be the image of $p_i$ under the quotient and $k_2$ the order of $q_1$ (equivalently, both $q_i$'s). 
    \end{itemize}
\end{defn}

\begin{lem}\label{lem-k=k1-H4}
    We have $k=k_2$.
\end{lem}

\begin{proof}
    Since $p_i$ is sent to $q_i$, we have $k_2\mid k$. Since $q_1^{k_2}=\id$, there exists $n\in\mathbb N$ such that $p_1^{k_2}=p_4^n$. Now, every $\sigma_i$ either commutes with $\sigma_1$ or with $\sigma_4$. Thus, the element $p_1^{k_2}=p_4^n$ is central in $P$. By Lemma \ref{lem-SNAQ-is-centreless}, we get $p_1^{k_2}=\id$. Thus, $k|k_2$ and $k=k_2$. 
\end{proof}

\begin{lem}\label{lem-size-K3-H4}
    We have $|K_3|\geq k^2$. 
\end{lem}

\begin{proof}
 Observe that $\langle q_1\rangle\cap\langle q_2\rangle=\{1\}$. Indeed, given $0\leq i,j\leq k-1$, if $q_1^i=q_2^j$, there exists $n\in\mathbb N$ such that $p_1^i=p_2^jp_4^n$. Now, every $\sigma_i$ either commutes with $\sigma_1$ or with $\sigma_2$ and $\sigma_4$. Thus, the element $p_1^i=p_2^jp_4^n$ is central in $P$. By Lemma \ref{lem-SNAQ-is-centreless}, we get $p_1^i=\id$. Since $o(p_1)=k$, this shows that $i=0$. Thus, we get $q_2^j=\id$ and again $j=0$ by Lemma \ref{lem-k=k1-H4}. Since $\langle q_1\rangle\cdot\langle q_3\rangle\subset K_3$, we obtain 
 \begin{equation*}
     |K_3|\geq|\langle q_1\rangle\cdot\langle q_2\rangle|=\frac{|\langle q_1\rangle||\langle q_2\rangle|}{|\langle q_1\rangle\cap\langle q_2\rangle|}=k^2.
 \end{equation*}

This concludes the proof.
\end{proof}

\begin{lem}\label{lem-size-K1-H4}
    We have $|K_1|\geq k^3$. In particular, if $k\geq 7$, we have $|K_1|>240$. 
\end{lem}

\begin{proof}
    We have $\langle p_4\rangle\cdot K_2\subset K_1$. Thus, we get
    \begin{equation*}
        \begin{aligned}
            |K_1|\geq |\langle p_4\rangle \cdot K_2|&=|\langle p_4\rangle|\left |\frac{K_2}{\langle p_4\rangle\cap  K_2}\right |\\
            &=k|K_3|.
        \end{aligned}
    \end{equation*}
    We conclude that the claimed inequality holds by Lemma \ref{lem-size-K3-H4}. For $k\geq 7$, this provides the lower bound $|K_1|\geq 7^3=343>240$. 
\end{proof}

Therefore, we can assume that $k\in\{3,4,5,6\}$. The following lemma handles these cases.

\begin{lem}\label{lem-size-P-k=3,4,5-H4}
    If $k\in\{3,4,5,6\}$, we have $|P|>7\, 200$.
\end{lem}

\begin{proof}
    In all these cases, the derived subgroup $(\A[H_4]/\llangle c^{15},\sigma_1^k\rrangle)'$ is perfect. Thus, the group $P'$ is perfect and non trivial, since $P$ is not abelian. We verify with GAP that the only case where a perfect group of order less than or equal to $7\, 200$ is a quotient of $(\A[H_4]/\llangle c^{15},\sigma_1^k\rrangle)'$ is $k=6$, where the groups $\mathfrak A_5$ and $\mathfrak A_5\times \mathfrak A_5$ are such quotients. Combining Lemmas \ref{lem-size-conjugacy-H4} and \ref{lem-size-K1-H4}, if $k=6$ we know that $|P|\geq 30\times 6^3=6\, 480$. Moreover, the abelianization of $(\A[H_4]/\llangle c^{15},\sigma_1^6\rrangle)'$ is isomorphic to $\Z/6$ and is generated by the image of the Artin generators. In particular, since the order of these generators is assumed to be $6$ in $P$, we have $|P|=6|P'|$. We deduce that $|P'|\geq 6\, 480/6=1\, 080$. Therefore, the group $P'$ cannot be isomorphic to $\mathfrak A_5$ and if it is isomorphic to $\mathfrak A_5\times \mathfrak A_5$, we have $|P|=6|\mathfrak A_5\times \mathfrak A_5|=21\, 600>7\, 200$. 
    In all cases, we have $|P|>7\, 200$. This concludes the proof. 
\end{proof}

\begin{proof}[Proof of Proposition \ref{prop-SNAQ-H4}]
    Lemmas \ref{lem-size-K1-H4} and \ref{lem-size-P-k=3,4,5-H4} prove Proposition \ref{prop-order-centralizer-H4}, which in turn implies Proposition \ref{prop-SNAQ-H4}.
\end{proof}

\subsection{Proof of Theorem \ref{thm-SNAQ}}

The different statements of Theorem \ref{thm-SNAQ} are proven separately in Propositions \ref{prop-SNAQ-firstcases},\ref{prop-SNAQ-E6},\ref{prop-SNAQ-E7},\ref{prop-SNAQ-E8},\ref{prop-SNAQ-H3} and \ref{prop-SNAQ-H4}.

\printbibliography
\end{document}